\documentclass[12pt,oneside]{amsart} 
\usepackage{amssymb,amscd}
\usepackage{euscript}
 
 \usepackage{color} 

      \theoremstyle{plain}
      \newtheorem{theorem}{Theorem}[section]
      \newtheorem{lemma}[theorem]{Lemma}
      \newtheorem{corollary}[theorem]{Corollary}
      \newtheorem{proposition}[theorem]{Proposition}
      \newtheorem{remark}[theorem]{Remark}
      
      \newtheorem{definition}[theorem]{Definition}

\numberwithin{equation}{section}

      \makeatletter
      \def\@setcopyright{}
      \def\serieslogo@{}
      \makeatother

\def\A{\EuScript{A}} 
\def\B{\EuScript{B}} 
\def\D{\EuScript{D}} 
 
\def\E{\mathcal{V}}
\def\V{\mathcal{V}}
\def\n{\mathcal N}
\def\M{\mathcal M}
\def\H{\mathcal H}
\def\T{\mathcal{T}}
\def\R{\mathbb R}
\def\Z{\mathbb Z}
\def\N{\mathbb N}

\def\Id{\text{Id}}
\def\e{\epsilon}
\def\a{\alpha}
\def\b{\beta}
\def\la{\lambda}

\def\P{\Phi}
\def\bv{\mathbf v}

\def\tw{\tilde W}
\def\tg{\tilde g}

\def\dpm{\mbox{Diff}^{\,p}(\M)}
\def\drm{\mbox{Diff}^{\,r}(\M)}
\def\dqm{\mbox{Diff}^{\,q}(\M)}

\def\dlm{\mbox{Diff}^{\,\ell}(\M)}
\def\dpsm{\mbox{Diff}^{\,p'}(\M)}
\def\dzm{\mbox{Diff}^{\,0}(\M)}
\def\dom{\mbox{Diff}^{\,1}(\M)}

\begin{document}

\author{Victoria Sadovskaya$^{\ast}$}

\address{Department of Mathematics, The Pennsylvania State University, University Park, PA 16802, USA.}
\email{sadovskaya@psu.edu}

\title[Diffeomorphism cocycles over partially hyperbolic systems]
{Diffeomorphism cocycles over partially hyperbolic systems} 

\thanks{$^{\ast}$ Supported in part by NSF grant DMS-1764216}


\begin{abstract} 

We consider H\"older continuous cocycles over an accessible partially hyperbolic system 
with values in the group of diffeomorphisms of a compact manifold $\M$. 
We obtain several results for this setting.
If a cocycle is bounded in $C^{1+\gamma}$, we show that it has a continuous invariant family 
of $\gamma$-H\"older Riemannian metrics on $\M$. 
We establish continuity of a measurable conjugacy between two  cocycles
assuming bunching or existence of holonomies for both and pre-compactness in $C^0$ for one of them. 
We give conditions for existence of a continuous conjugacy 
between two cocycles in terms of their cycle weights.
We also study the relation between
 the conjugacy and holonomies of the cocycles.  Our results give arbitrarily small loss of 
regularity of the conjugacy along the fiber compared to that of the holonomies and of the cocycle.
\end{abstract}

\maketitle 


\section{Introduction and statement of the results}

Cohomology of group-valued cocycles over hyperbolic and, later, partially hyperbolic systems 
has been extensively studies starting with the work of Liv\v{s}ic \cite{Liv1, Liv2},
where he obtained definitive results for commutative groups and 
some results for more general groups.  The theory has 
many applications to rigidity of hyperbolic and partially hyperbolic systems and actions.
The case of non-commutative groups such as  $GL(n,\R)$ is more complicated, and groups of diffeomorphisms present further difficulties. The study of diffeomorphism-valued cocycles over hyperbolic systems 
 was started in \cite{NT95,NT96},  and continued in \cite{NT98,pKtN,LW,LW11,ASV,BK,KP,AKL,S19,DX}.

In this paper we study cohomology of diffeomorphism-valued cocycles over accessible 
partially hyperbolic systems. The central question in this area
is existence and regularity of a conjugacy, or transfer map, between two cocycles.  
One of our main theorems gives conditions for continuity of a measurable conjugacy
between two cocycles.  Such a result  is new even for hyperbolic systems in the base. 
It yields, in particular, continuity of a measurable conjugacy 
to the identity for any bunched cocycle. Another theorem shows that 
bounded cocycles are isometric,
which extends our recent results in \cite{S19} to partially hyperbolic systems.
We also give conditions for existence of a conjugacy between two arbitrary cocycles 
in terms of their cycle weights. Results of this type were established in \cite{pKtN}
for cohomology  to a constant cocycle, and used to obtain certain cocycle rigidity for higher rank  
 hyperbolic abelian group actions in \cite{pKtN,DX}. 

In this paper we study cocycles depending H\"older continuously on the base point.
When the dependence on the base point is smooth, one can apply the theory 
of smooth partially hyperbolic  systems  to the skew product, as in \cite{NT98,pKtN,DX}.
Our approach is different, and an important role in the arguments 
is played by the {\em holonomies} of the cocycles and their relation with a conjugacy.
We also use results from \cite{ASV} on continuity of invariant sections of fiber bundles
over partially hyperbolic systems.
In our setup, it is  important to consider the regularity of the conjugacy along the fiber,
which may be lower than that of the cocycles.  
For a conjugacy to the identity cocycle, the regularity can be bootstrapped  
to that of the cocycle \cite{LW11}, but there are no such results for two general cocycles.
We obtain a conjugacy almost as regular as the holonomies of the cocycle,
which, in turn are almost as regular as the cocycle satisfying sufficient bunching.
\vskip.1cm
Now we formulate the main definitions and results.

\subsection{Basic definitions}
Let $X$ be a compact connected manifold.
 A diffeomorphism $f$ of $X$ is {\em partially hyperbolic} if
there exist a nontrivial $Df$-invariant splitting of the tangent bundle 
$TX =E^s\oplus E^c \oplus E^u$,  a Riemannian 
metric on $X$, and  positive  continuous 
functions $\la<1,\,$ $\hat\la<1,\,$ $\xi,$ $\hat\xi\,$ such that 
for any $x \in X$ and any unit vectors  
$\,\bv^s\in E^s(x)$, $\,\bv^c\in E^c(x)$, and $\,\bv^u\in E^u(x)$
\begin{equation}\label{partial def}
\|Df_x(\bv^s)\| < \la(x) <\xi(x) <\|Df_x(\bv^c)\| < \hat\xi(x)^{-1} <
\hat\la(x)^{-1} <\|Df_x(\bv^u)\|.
\end{equation}
 The sub-bundles $E^s$, $E^u$, and $E^c$ are called 
 stable, unstable, and center. $E^s$ and $E^u$  are 
tangent to the stable and unstable foliations $W^s$ and $W^u$, respectively.  
 If the center bundle is trivial, $f$ is called  {\it Anosov}.
 
 The diffeomorphism $f$ is  {\em center bunched}\,
if the functions  $\la, \hat\la, \xi, \hat\xi$  can be 
chosen so that $\la<\xi \hat \xi $ and $\hat\la<\xi \hat \xi$.

 An {\it $su$-path} in $X$ is a concatenation 
 of finitely many subpaths which lie entirely in a single 
 leaf of $W^s$ or  $W^u$. The
 diffeomorphism $f$  is called {\em accessible}  if any two points 
 in $X$ can be connected by an $su$-path.
 
We say that $f$ is {\em volume-preserving} if it has an invariant probability 
measure  $\mu$ in the measure class of a volume induced by a 
Riemannian metric. It was proved in \cite{BW}  that any essentially accessible 
center bunched $C^2$ partially hyperbolic diffeomorphism $f$ is ergodic with respect to such  $\mu$.

\begin{definition} Let $f$ be a homeomorphism of a compact metric space $X$
and let $A$ be a function from $X$ to $\dqm$. 
The {\em $\dqm$-valued cocycle over $f$ generated by }$A$ 
is the map $\A:\,X \times \Z \,\to \dqm$ defined  by $\,\A(x,0)=\Id\,$ and for $n\in \N$,
 $$
\A(x,n)=\A_x^n = A(f^{n-1} x)\circ \cdots \circ A(x) \quad\text{and}\quad\,
\A(x,-n)=\A_x^{-n}= (\A_{f^{-n} x}^n)^{-1} .
$$
\end{definition}

\noindent Clearly, $\A$ satisfies the {\em cocycle equation}\,
$\A^{n+k}_x= \A^n_{f^k x} \circ \A^k_x$.
\vskip.1cm

In this paper we consider the group $\dzm$ of homeomorphisms of $\M$ and 
the groups of diffeomorphisms $\dqm$, $q\ge1$.  We denote by $\| . \|_{C^q}$ the usual $C^q$ norm adapted to the manifold setting,   set $|g|_{C^q} = \|g \|_{C^q}+ \| g^{-1} \|_{C^q}$,  and consider a 
distance $d_{C^q}$ on  $\dqm$, see Section \ref{Crdist}. 

\vskip.1cm
We say that a $\dqm$-valued cocycle $\A$ is {\em $\beta$-H\"older}, $0<\beta \le 1$, if 
 there exists a constant $c>0$ such that
\begin{equation}\label{Holder A}
 d_{C^q} (\A_x, \A_y) \le  \, c\, d_X (x,y)^\beta \quad\text{for all  }x,y \in X. 
\end{equation}

 
\subsection{H\"older continuous cocycles with bounded set of values are isometric.}
The following theorem gives a partially hyperbolic version of Theorem 1.3 in \cite{S19},
where we considered Diff$^2(\M)$-valued cocycles with bounded periodic data
over hyperbolic systems.

\begin{theorem} \label{isometric} 

Let $f:X\to X$ be an accessible  center bunched $C^2$ partially hyperbolic diffeomorphism preserving a volume $\mu$.

Let  $0<\gamma\le 1$, and let  $\A$  be $\beta$-H\"older continuous 
Diff$^{\,1+\gamma}(\M)$-valued cocycle over $(X,f)$ such that the set 
$\{\,|\A_x^n|_{C^{1+\gamma}}:\, x\in X, \; n\in \Z\,\}$ is bounded.
Then there exists a family of Riemannian metrics $\{\tau_x: \, x\in X\}$ on $\M$
such that 
\begin{itemize}
\item[{\bf (a)}] $\A_x : (\M,\tau_x)\to (\M,\tau_{fx})$ is an isometry for each $x\in X,$ 
\vskip.1cm
\item[{\bf (b)}] Each $\tau_x$ is $\gamma$-H\"older continuous on $\M$, and
\vskip.1cm
\item[{\bf (c)}] $\tau_x$ depends continuously on $x$ in $C^{\a}$ distance  for each $0<\alpha< \gamma$.
\vskip.1cm
\item[{\bf (d)}]  For each $0<\alpha< \gamma$,\, $\tau_x$ depends H\"older continuously on $x$ along the leaves of $W^s$ and $W^u$ 
in $C^{\a}$ distance with exponent $\beta(\gamma-\alpha)$.
\end{itemize}

\noindent If $(X,f)$ is a hyperbolic system  and a cocycle $\A$ is as above, then, additionally,
  for each $0<\alpha< \gamma$ the metric $\tau_x$ depends {\em H\"older} continuously on $x\in X$ in $C^{\a}$ distance with exponent $\beta(\gamma-\alpha)$.
\end{theorem}

In the theorem above, by a hyperbolic system we mean either a transitive Anosov diffeomorphism, 
or a mixing diffeomorphism of a locally maximal hyperbolic set, or 
a mixing subshift of finite type, see  \cite[Section 2.1]{S19} or \cite{KH} for definitions.


\subsection{Existence and regularity of holonomies}
An important role in the study of non-commutative cocycles, 
and in particular diffeomorphism-valued cocycles is played by their holonomies.
For example, in Theorem \ref{isometric}, essential invariance of $\tau$ under holonomis 
helps to obtain its regularity.

\begin{definition} \label{def hol}
Let $(X,f)$ be a hyperbolic or partially hyperbolic system, 
and let $\A$ be a $\dqm$-valued cocycle over $(X,f)$.
 We say that $\A$ has {\em stable holonomies} $H^{\A,s}_{x,y}$ in $\drm$ if  
  $$
  H^{\A,s}_{x,y}=\lim_{n\to+\infty} (\A^n_y)^{-1} \circ \A^n_x \;\text{ exists in $\drm$
  for every $x\in X$ and $y\in W^s(x)$},
$$
 and the map $(x,y)\mapsto H^{\A,s}_{x,y}$  into $\drm$
is continuous on the set of pairs $(x,y)$ where $x\in X$ and $y\in W_{loc}^s(x)$.
If $r=0$, we also require that the homeomorphisms $H^{\A,s}_{x,y}$ 
  are H\"older continuous with uniform exponent and constant  for all  such pairs $(x,y)$.

The {\em unstable holonomies} $H^{\A,u}_{x,y}$ are similarly defined as
$$
H^{\A,u}_{x,y}=\lim_{n\to-\infty} (\A^n_y)^{-1} \circ \A^n_x,\quad \text{where }y\in W^u(x).
$$
We say that $\A$ has holonomies if it has both stable and unstable ones.
\end{definition}

Clearly,  $H^{\A,\,s/u}_{x,x}=\Id$ for every $x\in X$.
We say that the stable holonomies of $\A$ are {\em $\beta'$-H\"older  along
 the stable leaves} if for some $c_1>0$,
\begin{equation} \label{beta'}
 d_{C^r}(H^{\A,\, s}_{\,x,y},\Id) \leq c_1\,d_X (x,y)^{\beta'} 
 \text{ for all  $x\in X$   and $y\in W^{s}_{loc}(x)$,}
\end{equation} 
H\"older continuity of  $H^{\A,u}$ along the unstable leaves is defined similarly.

\vskip.2cm 
Existence of holonomies for cocycles has been extensively studied. We summarize the results for homeomorphism and diffeomorphism valued cocycles. They show that H\"older continuity and 
certain bunching, or domination, assumptions on the cocycle 
imply existence of its holonomies and their H\"older continuity along the stable and unstable leaves.
We formulate results for stable holonomies. The statements for unstable holonomies are similar,
with $\hat \la$ in place of $\la$.
\vskip.1cm

Let $(X,f)$ be a hyperbolic or partially  hyperbolic system, and let
\begin{equation}\label{lambda}
\la=\max\,\{ \la(x):\,x\in X \},  \,\text{ where $\la(x)$ is as in \eqref{partial def}.}
\end{equation}
Let $\A$ be a $\beta$-H\"older continuous $\dqm$-valued cocycle, $q\ge 1$, and let 
\begin{equation}\label{sigma}
\sigma = \max_{x\in X} \, \max \{\|D\A_x\|, \, \|D\A_x^{-1}\|\},
\;\text{ where }\, \|D\A_x\|=\max_{t\in \M}\|D_t\A_x\|.
\end{equation}

\begin{itemize}

\item[{\bf (E1)}]  \cite[Proposition 3.3]{pKtN},\,  \cite[Proposition 3.10]{ASV},\, [Proposition \ref{C0 hol}].\,\\
If  $q=1$ and $\sigma \la^\beta <1$, 
then $\A$ has stable holonomies in $\dzm$,
and they are $\beta$-H\"older along the leaves of $W^s$.
Instead of $\sigma$ as in \eqref{sigma}, we can take 
$\sigma$ such that for some constant $K$ 
$$
\| D \A^n_x\| \le K \sigma^{|n|}
\;\text{ for every $x\in X$ and $n\in \Z$}.
$$
\vskip.1cm

\item[{\bf (E2)}] \cite[Proposition 3.1]{BK}\,
If  $2 \le q\in \N$  such that $\sigma^{2q-1} \la^\beta <1$, 
then $\A$ has stable  holonomies in Diff$^{\,q-1}(\M)$, 
and they are $\beta$-H\"older along the leaves of $W^s$.
\vskip.2cm

\item[{\bf (E3)}] [Propositions \ref{holonomies}] \,
If $k\le r<q<k+1$, where $k\in \N$,  and there exist $\eta$ 
and $K$ such that $\,\eta^{2(r+ 1)/(q-r)} \cdot \la^{\beta}<1\,$ and 
 $\, |\A_x^n|_{C^q} \le K\eta^n\,$ for all $x\in X$  and $n\in \N$,
then $\A$ has stable holomomies in $\drm$, 
and they are $\beta (q-r)$\,-\,H\"older along the leaves of $W^s$. 
\vskip.2cm

\item[{\bf (E3$'$)}] [Remark \ref{from bunching}] \, If $k\le r<q<k+1$, where $k\in \N$, 
$\A$ is a Diff$^{\,k+1}(\M)$-valued cocycle with  bounded $|\A_x|_{C^{k+1}}$, and $\,\sigma^{2(r+1)(k+1)/ (q-r)} \cdot \lambda^{\beta}<1$,
then $\A$ has stable holomomies in $\drm$, 
and they are $\beta  (q-r)$\,-\,H\"older continuous along the leaves of $W^s$. 
\end{itemize}


\subsection{Continuity  of a measurable conjugacy between two cocycles}

We consider a {\em conjugacy}, or {\em transfer map}, between two cocycles.
If it exists, the cocycles are called {\em cohomologous}.

\begin{definition} \label{cohdef} Let $\A$ and $\B$ be $\dqm$-valued cocycles over $(X,f)$.
A  {\em conjugacy} between $\A$ and $\B$  is a 
 function  $\P:X\to \drm$ such that
\begin{equation}\label{conj eq}
  \A_x^n=\P_{f^n x} \circ \B_x^n \circ  \P_x^{-1} 
  \quad\text{ for all }n\in \Z \text{ and }x\in X,
\end{equation}
equivalently, $\A_x=\P_{fx} \circ \B_x \circ \P_x^{-1}\,\text{ for all }x\in X.$ 
\end{definition}

A conjugacy can be considered in various regularities, for example continuous, H\"older continuous, or measurable.
In the latter case we understand that $\P$ is defined and satisfies \eqref{conj eq} 
on a set of full measure.
 \vskip.1cm
 
 A key step in proving regularity of a measurable conjugacy is showing that it 
 intertwines the holonomies of the cocycles. 
  
\begin{definition} \label{def int}
Let $\A$ and $\B$ be cocycles with holonomies, and let $\P$ be a conjugacy between them.
We say that   $\P$ {\em intertwines the holonomies} of $\A$ and $\B$ on a set $Y\subseteq X$ if
\begin{equation}\label{intertwines def}
H_{x,y}^{\A,s/u}=\P_y\circ H_{x,y}^{\B,s/u}\circ \P_x^{-1}\quad \text{for all }x,y \in Y
\text{ such that }y\in W^{s/u}(x).
\end{equation}
\end{definition}

In the following theorem we establish continuity of a measurable conjugacy between cocycles 
over partially hyperbolic diffeomorphisms. If the holonomies of the cocycles are H\"older continuous
 along the leaves of $W^s$ and $W^u$, as we have in (E1-E3$'$), then the conjugacy is also 
 H\"older continuous along the leaves.
We note that without suitable assumptions a measurable conjugacy may not be continuous
 even if $f$ is an Anosov diffeomorphism and the cocycles are linear,  close to identity,
 and one of them is constant \cite[Section 9]{PW}.

\begin{theorem}\label{partially hyperbolic}
Let $f:X\to X$ be an accessible  center bunched $C^2$ partially hyperbolic diffeomorphism 
preserving a volume $\mu$.
Let  $\A$ and $\B$ be $\dqm$-valued  cocycles over $(X,f)$.
Suppose that the set $\{\B^n_x:\; x\in X, \; n\in \Z\}$
has compact closure in $\dzm$ and  that 
$\A$ and $\B$ have  holonomies in $\drm$, where either $r=0$ or $1\le r\le q$.
\vskip.15cm

\noindent {\bf (a)} Let $\P:X\to \drm$ be a $\mu$-measurable conjugacy between $\A$ and $\B$.
Then $\P$ coincides on a set of full measure with a bounded conjugacy $\tilde \P:X\to\drm $
which intertwines the holonomies of $\A$ and $\B$.
The function $\tilde \P:X\to \dpm$ is continuous for $p=r$ if $r$ is an integer,
and any $p<r$ otherwise.
\vskip.15cm

\noindent {\bf (b)}  Suppose that $r>1$ and the stable and unstable holonomies of $\A$ and $\B$ are $\beta'$-H\"older 
along the stable and unstable leaves respectively in the sense of \eqref{beta'}.
Then the conjugacy $\tilde \P:X\to \dpsm$ is H\"older continuous along
 the stable and unstable leaves with the exponent $\beta'(r-p')$ for any $p'$ such that 
  $r-1\le p'<r$  if  $r$ is an integer, and $\lfloor r\rfloor \le p'<r$ otherwise.
\end{theorem}

 By the results on existence of holonomies, 
instead of assuming existence and H\"older regularity 
of the holonomies we can assume that $\A$ and $\B$ are H\"older continuous 
cocycles with suitable bunching as in (E1-E3$'$).

\begin{remark} In the case of  $\B\equiv \Id$, all assumptions  on $\B$ are  satisfied,
and we obtain continuity of a measurable conjugacy to the identity cocycle.
Results of this type are often referred to as {\em measurable  Liv\v{s}ic theorems.}
\end{remark}

To the best of our knowledge,  it is the first result of this type for diffeomorphism-valued  cocycles 
even over a hyperbolic system and with $\B\equiv \Id$.

Theorem \ref{partially hyperbolic} applies to a volume-preserving Anosov diffeomorphism
since it is accessible by the local product structure of the stable and unstable manifolds,
and trivially center bunched. In this case we also obtain  H\"older continuity of $\Phi$
 on $X$.

\begin{corollary} \label{hyperbolic}
If $f$ in Theorem \ref{partially hyperbolic} is an Anosov diffeomorphism, $r>1$, 
and the stable and unstable holonomies of $\A$ and $\B$ are $\beta'$-H\"older 
along the stable and unstable leaves respectively, 
then $\tilde \P:X\to \dpsm$ is H\"older continuous on $X$ with the exponent $\beta'(r-p')$
for any $p'$ as above.

\end{corollary}

\vskip.2cm 


\subsection{Existence and properties of a conjugacy intertwining holonomies}$\;$\\
We begin with a result that a H\"older continuous conjugacy between 
sufficiently bunched cocycles  intertwines their holonomies. 

\begin{proposition} \label{C0 int}
Let $(X,f)$ be a hyperbolic or partially hyperbolic system, and let  $\la$ be as in \eqref{lambda}.
Let $\A$ and $\B$ be $\dom$-valued cocycles over $(X,f)$
so that  $\A_x, \B_x:X\to \dzm$ are $\beta$-H\"older continuous. 
If there exist constants $K$ and $\,\sigma$ such that 
$$
\sigma \la^\beta <1  \quad\text{and}\quad
\| D \A^n_x\| \le K \sigma^{|n|} , \;\;  \| D \B^n_x\| \le K \sigma^{|n|} 
\;\text{ for every $x\in X$ and $n\in \Z$},
$$
then $\A$ and $\B$ have stable holonomies in $\dzm$, and  any $\b$-H\"older continuous conjugacy  
$\P:X\to \dzm$ between $\A$ and $\B$ intertwines the stable holonomies.
\end{proposition}

We note that H\"older continuity of the conjugacy with exponent less than $\beta$
does not guarantee the intertwining, even for linear cocycles over hyperbolic systems, 
see \cite[Proposition 4.4]{KS16} based on examples in \cite{dlL,NT98}.

\vskip.1cm 

Since  intertwining of the  holonomies is a ``pointwise" property,  it suffices to
obtain it in the lowest regularity. Once the intertwining is established, further properties 
of the conjugacy can be obtained, as stated in Proposition \ref{prop of inter conj} below.
Also, if $f$ is an Anosov diffeomorphism and the cocycles are smooth along the base $X$,
the main result of \cite{NT98} can be used to obtain smoothness of an
intertwining conjugacy.

\vskip.2cm

Let $f:X\to X$ be an accessible partially hyperbolic diffeomorphism,  let 
$\A$ be a $\dqm$-valued cocycle over  $(X,f)$ with the stable and unstable holonomies
 $H^{\A,s}$ and $H^{\A,u}$.
An $su$-cycle in $X$ is a  closed $su$-path, which we view as a sequence of points
$$
P=P_{x_0}=\{x_0, x_1, \dots , x_{k-1}, x_k=x_0\}, \;\text{ where }\,x_{i+1}\in W^{s/u}(x_i),\; i=0, \dots, {k-1}.
$$
We define the {\em cycle weight} of  $P$  as
\begin{equation}\label{weight}
\H^{\A, P}_{x_0} =  H_{x_{k-1},x_k} \circ \dots \circ H_{x_1, x_2} \circ H_{x_0, x_1}, 
\end{equation}
where $H_{x_i, x_{i+1}}=H^{\A,s/u}_{x_i, x_{i+1}}$ if $x_{i+1} \in W^{s/u} (x_i)$.
One can similarly consider the weight $\H^{\A, P}_{x_0,x_k}$ for an  $su$-path 
$P=P_{x_0, x_k}$ from $x_0$ to $x_k$.

\begin{proposition} \label{prop of inter conj}
Let $f:X\to X$ be an accessible $C^1$ partially hyperbolic diffeomorphism.
Let $\A$ and $\B$ be $\dzm$-valued cocycles over $(X,f)$ with the stable and 
unstable holonomies  $H^{\A,\, s/u}$ and $H^{\B,\, s/u}$ in $\dzm$.
Let $\P:X\to \dzm$ be any conjugacy between  $\A$ and $\B$
which intertwines their holonomies. Then

\begin{itemize}
\item[{\bf (a)}]  $\P$ conjugates cycle weights, i.e., 
$\H^{\A,P_x}_x=\P_x\circ \H^{\B,P_x}_x \circ \P_x^{-1}$
for any $su$-cycle $P_x$;
\vskip.1cm

\item[{\bf (b)}] More generally,   
$\,\H^{\A,P}_{x,y}=\P_y\circ \H^{\B,P}_{x,y}\circ \P_x^{-1}$
 for any  $su$-path $P=P_{x,y}$;
\vskip.1cm

\item[{\bf (c)}]  $\P$ is uniquely determined by its value at one point;
\vskip.1cm

\item[{\bf (d)}]  $\P:X\to \dzm$ is continuous;
\vskip.1cm

\item[{\bf (e)}] If for some $r\ge 1$,\, $H^{\A,\, s/u}$ and $H^{\B,\, s/u}$ are in  $\drm$ 
and $\P_{x_0}\in \drm$ for some $x_0\in X$, then $\P$ is a bounded function from $X$ to $\drm$,
and $\P:X\to \dpm$ is continuous for $p=r$ if $r$ is an integer,
and for any $p<r$ otherwise.

 \end{itemize}
 
\end{proposition}

The next theorem gives a sufficient condition for existence of a continuous 
conjugacy intertwining holonomies. By the previous proposition, condition (b)
is also necessary.

\begin{theorem} \label{sufficient}
Let $f:X\to X$ be an accessible $C^1$ partially hyperbolic diffeomorphism.
Let  $\A$ and $\B$ be  $\dqm$-valued cocycles over $(X,f)$  
with holonomies  in $\drm$, where $r=0$ or $1\le r\le q$.

\vskip.1cm
\begin{itemize}

\item[{\bf (a)}] Suppose that  there exist a fixed point $x_0\in X$ and $\P_{x_0}\in  \drm$ such that 
\vskip.1cm
 (a1) $\H^{\A,P}_{x_0}=\P_{x_0}\circ \H^{\B,P}_{x_0} \circ \P_{x_0}^{-1}$
for every $su$-cycle $P_{x_0}$, and 
\vskip.1cm
(a2) $\A_{x_0}= \P_{x_0}\circ \B_{x_0} \circ \P_{x_0}^{-1}.$

\vskip.2cm
\item[{\bf (b)}] More generally, suppose that there exist $x_0\in X$ and $\P_{x_0}\in  \drm$  satisfying 
\vskip.1cm

(b1)=(a1) and
\vskip.1cm
(b2) $\A_{x_0}= \P_{f{x_0}}\circ \B_{x_0} \circ \P_{x_0}^{-1}$,
\,where 
$\,\P_{f{x_0}}=\H^{\A,\tilde P}_{{x_0},f{x_0}} \circ \P_{x_0}\circ  (\H^{\B,\tilde P}_{{x_0},f{x_0}})^{-1}$

\hskip.8cm for some $su$-path $\tilde P=\tilde P_{{x_0},f{x_0}}$ from $x_0$ to $fx_0$.
\end{itemize}
Then there exists a unique conjugacy $\P$  between $\A$ and $\B$ 
with value $\P_{x_0}$ at $x_0$ that intertwines $H^{\A}$ and $H^{\B}$.
The function $\P:X\to \drm$ is bounded and $\P:X\to \dpm$ is  continuous 
for $p=r$ if $r$ is an integer, and for any $p<r$ otherwise.

\end{theorem}

Case (a) can be viewed 
as a sufficient condition for extending a conjugacy from a given value at a fixed point.
The value $\P_{f{x_0}}$ in (b2) does not depend on the choice of a path $P_{x_0,fx_0}$
due to the first assumption. If $x_0$ is a fixed point for $f$, then for the trivial path from 
$x_0$ to $fx_0=x_0$ the condition in (b2)  becomes 
$\A_{x_0}=\P_{x_0} \circ \B_{x_0} \circ \P_{x_0}^{-1}$,
so (b) indeed generalizes (a).

\vskip.1cm

As a corollary of Theorem \ref{sufficient}(b) we obtain the following result on conjugacy
to a constant cocycle. A similar result was established in \cite[Proposition 5.6]{pKtN}
for cocycles which depend
smoothy on the base point over partially hyperbolic systems satisfying a stronger accessibility assumption.

\begin{corollary}  \label{cohomologous to constant}
Let $f:X\to X$ be an accessible $C^1$ partially hyperbolic diffeomorphism.
Let  $\A$  be  $\dqm$-valued cocycle over $(X,f)$  
with holonomies  in $\drm$, where $r=0$ or $1\le r\le q$.
Suppose that 
\begin{equation} \label{trivial H}
\H^{\A,P_{x_0}}_{x_0}=\Id \quad  \text{ for 
every $su$-cycle $P_{x_0}$ based at some  point ${x_0}\in X$}.
\end{equation}  
Then there exists a bounded conjugacy $\P:X\to \drm$ between $\A$ and a {\em constant
cocycle} such that  
$\P:X\to \dpm$ is  continuous 
for $p=r$ if $r$ is an integer, and for any $p<r$ otherwise, and 
$\P$ satisfies 
\begin{equation}\label{PP}
\P_y\circ \P_x^{-1} = H_{x,y}^{\A,s/u}\;\text{  for all $x,y \in X$ such that $y\in W^{s/u}(x)$.}
\end{equation}
In particular, if $\A_{x_0}=\Id$ at a fixed point $x_0$ and \eqref{trivial H} holds,
then $\A$ is conjugate to the identity cocycle via $\P$ as above with $\P(x_0)=\Id$.
\end{corollary}

If condition  \eqref{trivial H} holds for some $x_0\in X$, then it holds for every $x\in X$,
since by accessibility  for any $su$-cycle  based at $x$ one can consider 
a corresponding $su$-cycle based at $x_0$.

We note  that a constant cocycle  conjugate to $\A$ is not unique in general. 
Also, if $\A$ is conjugate to a constant cocycle via a conjugacy  intertwining 
their holonomies, then \eqref{PP} holds and \eqref{trivial H} follows.

\vskip.1cm

The paper is organized as follows. 
In Section \ref{dist est} we define distances on the spaces $\drm$ and 
give estimates for norms and distances between compositions of diffeomorphisms.
In Section \ref{sec  hol} we formulate and prove results on existence and properties of 
holonomies of $\dqm$-valued cocycles. In Section \ref{proof isometric} we prove Theorem
\ref{isometric}; in Section \ref{5} we prove Proposition \ref{prop of inter conj}, 
Theorem \ref{partially hyperbolic} and Corollary \ref{hyperbolic};
 and  in Section  \ref{6} we prove  Proposition \ref{C0 int}, Theorem \ref{sufficient},
and Corollary \ref{cohomologous to constant}.


\section{Distances on $\drm$ and  estimates}\label{dist est}

\subsection{Distances on the space of diffeomorphisms $\drm$} \label{Crdist}
(\cite[Section~5]{LW},  \cite[Section 2.2]{S19})\,
We fix a smooth background Riemannian metric and the corresponding distance 
$d_\M$ on $\M$.

We denote the space of homeomorphisms of $\M$ by $\dzm$, and for $g,h\in \dzm$  we set
\begin{equation}\label{d0}
d_0(g,h)= \max_{t\in \M}\, d_\M(g(t),h(t)) \quad\text{and}\quad
d_{C^0}(g,h)=d_0(g,h)+ d_0(g^{-1},h^{-1}).
\end{equation}

Now we consider $r\ge 1$. 
The $C^r$ topology on the group of diffeomorphisms  $\drm$ can
be defined using coordinate patches and the $C^r$ norm in the Euclidean space. 
For any $g\in \drm$,  $r\in \N=\{1,2, \dots \}$, we define its $C^r$ size  as
$$
  |g|_{C^r}= \|g\|_{C^r}+ \|g^{-1}\|_{C^r} , \quad \text{where }\; 
  \|g\|_{C^r}= \max_{t\in \M} d_\M(g(t),t)+ \max_{1 \le i \le r}\max_{t\in \M} \|D_t^ig\|, 
$$
where $D_t^i\,g$ is the derivative of $g$ of order $i$ at $t$, and its norm is defined as
the norm of the corresponding multilinear form from $T_t\M$ to $T_{g(t)}\M$ 
with respect the Riemannian metric. For $r=k+\a$ with $k\in \N$ and  $0<\a\le1^-$, where $\a=1^-$ means Lipschitz, the definition is similar with
$$
 \|g\|_{C^r}=\|g\|_{C^k} + 
 \sup \,\{\, \|D^k_t -D^k_{t'}\|\cdot d_\M(t,t')^{-\alpha} : \;t, t'\in \M, \; 0< d_\M(t,t') <\e_0 \}.
$$
Here $\e_0$ is chosen small compared to the injectivity radius of $\M$ so that
the tangent bundle is locally trivialized via parallel transport and the difference makes sense.

\vskip.1cm 
A natural distance $d_{C^r}(g,h)$ 
on $\drm$,  $r\in \N$,
was defined in \cite{LW}  as the infimum of the lengths of piecewise $C^1$ paths  
in $\drm$ connecting $h$ with $g$  and $h^{-1}$ with $g^{-1}$,
where the length of a path $p_s$ is 
$$
\ell_{C^r} (p_s) =\max_{0 \le i \le r}\max_{t\in \M}\int \|\tfrac d{ds} (D_t^i \, p_s) \|\,ds.
$$
For $r=k+\a$,  one also adds 
the corresponding H\"older term:
$$
\begin{aligned}
&\ell_{C^r} (p_s) =\ell_{C^k} (p_s) 
+ \max_{t\in \M}\int |\tfrac d{ds}  \| (D^k p_s) \|_{\a,t}|\,ds, \quad\text{where}\\
& \|D^k g \|_{\a,t}= \sup \,\{\, \|D^k_{t'} -D^k_{t}\|\cdot d_\M(t',t)^{-\alpha} : 
 \,\;t'\in \M, \;\, 0< d_\M(t,t') <\e_0 \}.
 \end{aligned}
$$
For sufficiently $C^r$-close diffeomorphisms, the distance $d_{C^r}(g,h)$ is Lipschitz equivalent 
to $ \|g-h\|_{C^r}+ \|g^{-1}-h^{-1}\|_{C^r}$, where  the difference is understood 
using local trivialization. Specifically,
there exist  constants $\kappa$ and $\delta_0>0$ depending only on $r$ and 
the Riemannian metric so that 
\begin{equation}\label{compare 1}
\begin{aligned}
&\kappa^{-1} d_{C^r}(g,h) \le  \|g-h\|_{C^r}+ \|g^{-1}-h^{-1}\|_{C^r} \le \kappa \, d_{C^r}(g,h), \;\text{ provided that }\\
&\text{either $\;d_{C^r}(g,h)< \delta_0 |g|_{C^r}^{-1}\;$ or 
$\;\|g-h\|_{C^r}+ \|g^{-1}-h^{-1}\|_{C^r}<\delta_0 |g|_{C^r}^{-1}$.}
\end{aligned}
\end{equation}


\subsection{Estimates of norms and distances}\,
Lemma \ref{dlLO lemma} follows directly from Proposition 5.5 in \cite{dlLO},
and Lemma \ref{distance}  relies on further results in that paper.

\begin{lemma}\cite{dlLO}\label{dlLO lemma} 
For any $r\ge 1$ there exists a constant  $M_r$ such that for any $h,g \in C^r(\M)$,
$$
  \| h\circ g\|_{C^r}  \le M_r \,  \|h\|_{C^r} (1+\| g\|_{C^r})^r .
  $$
\end{lemma}

\begin{lemma}\cite[Lemma 3.6]{S19}\label{distance} 
Let $q=k+\gamma$, $r=k+\a$, and $\rho=q-r$, where $k\in \N$ and $0\le\a<\gamma\le1^-.$
There exists a constant $M=M(r,q,\M,K)$ such that 
for any $g, \tg \in \dqm$ and $h_1, h_2\in \drm$  with $|h_1|_{C^r}, |h_2|_{C^r}\le K,$
\begin{equation} \label{dist}
\begin{aligned}
&d_{C^r} (g\circ h_1\circ \tg,\, g\circ h_2 \circ \tg) \le \\
&\le M \left( \| g\|_{C^q} (1+\| \tg\|_{C^r})^r + \| \tg^{-1}\|_{C^q} (1+ \| g^{-1}\|_{C^r})^r \right)
 \cdot d_{C^r}(h_1, h_2)^\rho
\end{aligned}
\end{equation} 
provided that $d_{C^r}(h_1, h_2)\le \delta_0 |h_1|_{C^r}^{-1}$ and  the right hand side 
of \eqref{dist} is less than 
\begin{equation} \label{delta 0}
\delta_0 \left( M_r^{2}\, (1+| h_1|_{C^r})^{r}(   \|g\|_{C^r}  \, (1+\| \tg\|_{C^r})^r +
 \|\tg^{-1}\|_{C^r}   \, (1+\| g^{-1}\|_{C^r})^r )\right)^{-1}, 
\end{equation} 
 where $ \delta_0$ is as in \eqref{compare 1}.  
\end{lemma}


\section{Existence and properties of the holonomies}\label{sec  hol}

In this section, we state and prove some results on existence and properties of the holonomies,
first in $\dzm$ and then in $\drm$, $r\ge1$.
We formulate the results for stable holonomies. We note that we only consider 
pairs of points lying on the same stable leaf, and so only the contracting property 
of the stable leaves plays a role.
The statements and proofs for the unstable holonomies are similar. 

The local stable manifold of $x$, $W^{s}_{loc}(x)$, is 
a ball in $W^s(x)$  centered at $x$ of a small radius $\rho$. We choose $\rho$ sufficiently small 
so that, in particular, for $\la$ as in \eqref{lambda}, 
\begin{equation}\label{exp conv}
d_X (f^nx, f^ny)\le \la^n d_X(x,y)  \quad\text{for all }x\in X,\; y\in W^s_{loc}(x),\text{ and }n\in \N.
\end{equation}
\vskip.1cm

Results similar to the next proposition can be found in \cite[Proposition 3.3]{pKtN} and 
\cite[Proposition 3.10]{ASV}. We state and prove the result, 
including H\"older continuity along the leaves, joint continuity, and uniqueness,
under a weaker assumption of H\"older continuity of $\A_x$ as a function into $\dzm$ 
rather than $\dom$.

\begin{proposition} \label{C0 hol}
Let $(X,f)$ be a hyperbolic or partially hyperbolic system, and let 
$\A$ be a  $\dom$-valued cocycle over $(X,f)$
so that $\A_x:X\to \dzm$ is $\beta$-H\"older continuous.
Suppose that there exist constants $K$ and $\,\sigma$ such that 
\begin{equation}\label{A bunch}
\sigma \la^\beta <1\;\text{  and }\;\,
\| D \A^n_x\| \le K \sigma^{|n|} 
\;\text{ for every $x\in X$ and $n\in \Z$}.
\end{equation}
Then for any $x\in X$ and $y\in W^s(x)$,
the limit 
$H^{\A,s}_{x,y}=\underset{n\to+\infty}{\lim} (\A^n_y)^{-1} \circ \A^n_x$
exists in $\dzm$ and satisfies 
\vskip.2cm
\begin{itemize}
\item[{\bf (H1)}$\;$] $H^{\A,s}_{x,x}=\Id\,$ and $\,H^{\A,s}_{y,z} \circ H^{\A,s}_{x,y}=H^{\A,s}_{x,z}$,\,\,
which imply $(H^{\A,s}_{x,y})^{-1}=H^{\A,s}_{y,x};$
\vskip.1cm
\item[{\bf (H2)}$\;$] $H^{\A,s}_{x,y}= (\A^n_y)^{-1}\circ H^{\A,s}_{f^nx ,\,f^ny} \circ \A^n_x\;$ 
for all $n\in \N$;
\vskip.1cm
\item[{\bf (H3$^0$)}]  There exists a constant $c$ such that 
$$
d_{C^0}(H^s_{x,y},\Id)  \le  c\, d_X(x,y)^{\beta}\quad\text{ for all $x\in X$ and $y\in W^s_{loc}(x)$};
$$

\item[{\bf (H4$^0$)}]  The map $(x,y)\mapsto H^{\A, s}_{x,y}$ into $\dzm$ is continuous 
on the set of pairs  $(x,y)$,
where $x\in X$ and $y\in W^{s}_{loc}(x)$.
\vskip.2cm

\item[{\bf (H5)}$\;$]   The  homeomorphisms $H^{\A,s}_{x,y}$ 
  are H\"older continuous with uniform exponent   and
   constant  for all  pairs $(x,y)$ as above.
\end{itemize}
\vskip.1cm

\noindent Moreover, a family of maps $\{H_{x,y}^{\A,s}: \; x\in X, \; y\in W^s(x)\}$ in $\dzm$ satisfying 
{\em (H2)}  and {\em (H3$^0$)} is unique.

\end{proposition}

\noindent With $\la =\max_{x\in X}  \hat \la(x)$, a similar result holds for the unstable holonomies that satisfy
the following in place of (H2):
\begin{itemize}

\item[{\bf (H2$'$)}]  $H^{\A,u}_{x,\,y}=
(\A^{-n}_y)^{-1} \circ H^{\A,u}_{f^{-n}x ,\,f^{-n}y}  \circ \A^{-n}_x\;$ 
 for all $n\in \N$.
 \end{itemize}

\begin{remark}

Sometimes, holonomies of a cocycle are defined as {\em any} family of homeomorphisms 
$\tilde H^{A,\, s}_{x,y}$ satisfying  properties {\em (H1, H2, H4$^0$)}, and {\em (H5)} 
for non-linear cocycles, see e.g. \cite{ASV},
and the holonomies $H^{A,s}$ as in Definition \ref{def hol} are referred to as {\em standard 
holonomies} to distinguish them \cite{KS16}. Without condition {\em (H3$^0$)}
uniqueness of holonomies may fail even for linear 
cocycles, as discussed  in \cite{KS16} after Corollary~4.9. 
\end{remark}

\begin{proof}
The maps $H^s_{x,y}$ are constructed for $x\in X$ and $y\in W_{loc}^s(x)$,
and the extended to the whole leaf $W^s(x)$ by the invariance property (H2).
We fix $x$ and $y\in W_{loc}^s(x)$ and consider the sequence of diffeomorphisms 
$( (\A^n_y)^{-1}\circ  \A^n_x )_{n\ge 0}$.

For $g_1,g_2\in \dzm$, we consider $ d_0 (g_1,g_2)$ and $d_{C^0} (g_1,g_2)$
as in \eqref{d0}.
We write $x_n$ for $f^nx$ and $y_n$ for $f^ny$.
Since $\A$ is $\beta$-H\"older continuous 
and \eqref{exp conv} holds, we obtain
\begin{equation}\label{dAA}
\begin{aligned}
& d_0  (\A_{y_n}^{-1} \circ \A_{x_n}, \, \Id)= 
d_0  (\A_{y_n}^{-1} \circ \A_{x_n}, \, \A_{x_n}^{-1} \circ \A_{x_n})
= d_0  (\A_{y_n}^{-1}, \, \A_{x_n}^{-1}) \le \\
&\le d_{C^0}  (\A_{y_n}, \, \A_{x_n}) \le K_1 \, d_X(x_n,y_n)^\b 
\le K_1 \, d_X (x,y)^\b \cdot \la^{n\b}.
\end{aligned}
\end{equation}
 It follows that
 $$
 \begin{aligned}
 & d_0 \left( (\A_y^n)^{-1}  \circ \A_x^n ,\, (\A_y^{n+1})^{-1} \circ \A_x^{n+1}\right)=\\
 &= d_0\left( (\A_y^n)^{-1}\circ\Id\circ \A_x^n ,\,  (\A_y^n)^{-1}
 \circ(\A_{y_n})^{-1} \circ \A_{x_n}\circ \A_x^n\right) =\\
&=d_0\left( (\A_y^n)^{-1}\circ\Id ,\,  (\A_y^n)^{-1}
 \circ((\A_{y_n})^{-1} \circ \A_{x_n} \right)) \le  
 K \sigma^n\cdot d_0  (\A_{y_n}^{-1} \circ \A_{x_n}, \, \Id) \le \\
 & \le K \sigma^n\cdot K_2 \, d_X (x,y)^\b \cdot \la^{n\b} 
 = K_3\, d_X (x,y)^\b\cdot \theta^n,\;\text{ where }\, \theta = \sigma \la^\beta <1.
 \end{aligned}  
 $$
 The same estimate holds for $d_0$ between the inverses, and so there exists 
a constants $K_4$ such that for all $x\in X$ and $y\in W^s_{loc}(x)$,
 $$
d_{C^0} \left( (\A_y^n)^{-1}  \circ \A_x^n ,\, (\A_y^{n+1})^{-1} \circ \A_x^{n+1}\right) 
\le K_3\, d_X (x,y)^\b\cdot \theta^n.
$$ 
Thus \,$( (\A^n_y)^{-1}\circ  \A^n_x )$ is a Cauchy sequence in $\dzm$, 
and so it has a limit $H^{\A,s}_{x,y}$ there. The convergence is uniform 
in all  $(x,y)$ with $y\in W^s_{loc}(x)$.
For each $n\in \N$, $(\A^n_y)^{-1} \circ \A^n_x$ depends continuously on $(x,y)$,
hence so does the limit $H^{\A,s}_{x,y}$ and we obtain (H4$^0$).
\vskip.1cm

Properties (H1) and (H2) are easy to verify. For (H3$^0$), 
we note that  $(\A^0_y)^{-1} \circ \A^0_x=\Id,$\,  so for every $n\in \N$ we have
$$
 \begin{aligned}
& d_{C^0}    \left( (\A^n_y)^{-1} \circ \A^n_x, \, \Id \right) 
\le \,\sum_{i=0}^{n-1} d_{C^0}   \left( (\A_y^i)^{-1}  \circ \A_x^i ,\, (\A_y^{i+1})^{-1} \circ \A_x^{i+1}\right) \le \\
& \le  K_3\, d_X(x,y)^{\beta} \cdot \sum_{i=0}^\infty \theta^i \le  c\, d_X(x,y)^{\beta},
\quad\text{and hence }\;d_{C^0}(H^s_{x,y},\Id)  \le  c\, d_X(x,y)^{\beta}.
\end{aligned}
$$

A proof of (H5$^0$) is given in \cite[Proposition 3.10]{ASV}, and it uses only the assumptions 
of this proposition.

Finally, we prove the uniqueness.  Let $H=\{H_{x,y}\}$ and $\tilde H=\{\tilde H_{x,y}\}$ be two such families.
 By property (H2) it suffices to verify that $H_{x,y}= \tilde H_{x,y}$ for any $x\in X$ and 
 $y\in W^s_{loc}(x)$. We fix such $x$ and $y$. Then using property (H2), 
 assumption \eqref{A bunch}, and property (H3$^0$) we obtain
 $$
 \begin{aligned}
   & d_0 (H_{x,y}, \tilde H_{x,y}) = 
   d_0 ((\A^n_y)^{-1}\circ H_{x_n,\,y_n} \circ \A^n_x, \,\,
       (\A^n_y)^{-1}\circ \tilde H_{x_n,\,y_n} \circ \A^n_x)=\\
    &  =  d_0 ((\A^n_y)^{-1}\circ H_{x_n  ,\,y_n} , \,\,
       (\A^n_y)^{-1}\circ \tilde H_{x_n,\,y_n} ) \le 
       K\sigma^n \cdot  d_0 (H_{x_n,\,y_n}, \,\, \tilde H_{x_n ,\,y_n} ) \le \\
      &  \le K\sigma^n \cdot 2c(d_X(x_n ,y_n))^\beta= K\sigma^n \cdot 2c (d_X(x,y)\la^n)^\beta=
      K' (\sigma \la^\beta)^n \to 0 \;\text{ as }n\to\infty
\end{aligned}
 $$
since $\sigma \la^\beta <1$. Thus $H_{x,y}=\tilde  H_{x,y}$.
\end{proof}


The next proposition establishes existence and regularity of holonomies in $\drm$,
where $r\ge 1$.

\begin{proposition}  \label{holonomies} 
Let $\A$ be a $\beta$-H\"older $\dqm$-valued cocycle over a hyperbolic or partially 
hyperbolic system, 
where $q=k+\gamma$,  with  $k\in \N$ and $0<\gamma\le1^-$. 
Let  $k\le r<q$ and let $\rho = q-r$. Suppose that there exist constants 
$\eta$ and $K$ such that
\begin{equation} \label{C^q bunching}
 \eta^{2(r+1)} \la^{\beta\rho}<1 \quad\text{and}\quad
  |\A_x^n|_{C^q} \le K \eta^n \,\text{ for all } x\in X \text{ and }n\in \N .
\end{equation}
Then for any $x\in X$ and $y\in W^s(x)$, the limit 
 $H^{\A,s}_{x,y}=\underset{n\to+\infty}{\lim} (\A^n_y)^{-1} \circ \A^n_x\,\,$
exists 
in $\drm$ and satisfies {\em (H1)} and {\em (H2)} as in Proposition \ref{C0 hol}, and 
\vskip.2cm
\begin{itemize}
\item[{\bf (H3$^r$)}] There exists a constant $c_1$ such that 
$$
d_{C^r}(H^{\A,s}_{\,x,y},\Id) \leq c_1\,d_X (x,y)^{\beta \rho} \quad\text{ for all $x\in X$ and $y\in W^s_{loc}(x)$};
$$

\item[{\bf  (H4$^r$)}] The map $(x,y)\mapsto H^{\A, s}_{x,y}$ into $\drm$ is continuous 
on the set of pairs $(x,y)$, 
where $x\in X$ and $y\in W^{s}_{loc}(x)$.
\end{itemize}

\end{proposition}

\noindent Under the same assumptions, a similar result holds for the unstable holonomies.
\vskip.1cm

Parts (H1), (H2),  and  (H3$^r$) were established for cocycles over hyperbolic systems 
in Proposition 3.3 and Remark 3.4 in \cite{S19}. The same argument 
applies in the partially hyperbolic case since it involves only points on a stable manifold.
The key estimate is that for all 
sufficiently close $x,y$ with $y\in W^s_{loc}(x)$ and all $n\in\N$,
$$
d_{C^r} \left( (\A_y^n)^{-1}  \circ \A_x^n ,\, (\A_y^{n+1})^{-1} \circ \A_x^{n+1}\right)
 \le  K'\, d_X(x,y)^{\beta\rho} \cdot \tilde \theta^n, \;\text{ where }\tilde \theta<1.
$$
This yields that  $((\A_y^n)^{-1}  \circ \A_x^n)$ converges  to 
$H^{\A,s}_{x,y}$ in $\drm$ uniformly in such $(x,y)$.
Hence $H^{\A,s}_{x,y}$ depends continuously on $(x,y)$ and we obtain (H4$^r$).
\vskip.2cm

\begin{remark} \label{from bunching}
Suppose that $k\in \N$ and  $\A$ is a  Diff$^{\,k+1}(\M)$-valued 
cocycle with  bounded $|\A_x|_{C^{k+1}}$.  Then condition \eqref{C^q bunching} can be deduced 
from the bunching assumption that
$$
\sigma=\max_{x\in X}\, \max \{\|D\A_x\|,\, \|D(\A_x)^{-1}\|\}
\,\text{ satisfies } \,\sigma^{2(r+1)(k+1)/\rho} \cdot \lambda^\beta<1.
$$
\end{remark}

\noindent Indeed, By Lemma 5.5 in \cite{LW}, 
there exists a constant $c$ such that 
$$
\|D^m \A_x^n\| \le c \, \sigma^{m|n|}\;\text{ for every $x\in X$ and $1\le m\le k+1$.} 
$$
Since $|\A_x|_{C^0}$ is bounded, it follows that there exists a constant $c'$ such that 
$$
|\A_x^n|_{C^q} \le c'|\A_x^n|_{C^{k+1}}\le c''\sigma^{(k+1)n}\;\text{  for every $x\in X$ and $n\in \N$.}
$$
If $\,\sigma^{2(r+1)(k+1)/\rho} \cdot \lambda^\beta<1$,\, then 
$\,\sigma^{2(r+1)(k+1)} \cdot \lambda^{\beta\rho}<1$, and so  
condition \eqref{C^q bunching} is satisfied for $\eta=\sigma^{k+1}$.


\section{Proof of Theorem \ref{isometric}} \label{proof isometric}

The first part of the proof is essentially  the same for the hyperbolic 
and partially hyperbolic cases, and it follows the arguments in \cite{S19}, so we just outline it. 
We will give a detailed last part of the proof in
the partially hyperbolic case.
\vskip.1cm

Let $(X,f)$ be a hyperbolic or partially hyperbolic system.
We consider the  vector bundle  $\E$  over $X\times \M$ with fiber
$\E_{(x,t)} =T_t\M$ and the linear cocycle 
$$
\D_{(x,t)}= D_t\A_x \;\text{ on $\E\;$ over the skew product }\;F(x,t)=(f(x), \A_x (t)).
$$
The iterates of $\D$ are 
given by  $\D^n_{(x,t)}: T_{t\,}\M \to T_{\A^n_x(t)\,}\M$, where
$
\D^n_{(x,t)}= D_t\A_x^n $.
Since  the value set of the  cocycle $\A$ is bounded in Diff$^{\,1+\gamma}(\M)$,
$\|\D_{(x,t)}^n\|$ is uniformly bounded in $(x,t)\in X\times \M$ and $n\in \Z$, and 
there exists a constant $c_2$ such that 
$$
\|  \D^n_{(x,t)} - \D^n_{(x,t')} \| \le c_2\, d_\M (t, t')^\gamma \quad\text{for all nearby }t,t'\in \M \text{ and }n\in \Z.
$$

\vskip.2cm
The space $\T^m$  of inner products on $\R^m$
 identifies with the space of real symmetric positive definite $m\times m$ matrices, 
which is isomorphic to $GL(m,\R) /SO(m,\R)$. 
The group $GL(m,\R)$ acts transitively on $\T^m$
via $A[E] = A^T E \, A.$  
The space $\T^m$  with a certain $GL(m,\R)$-invariant 
metric is a Riemannian symmetric space 
of non-positive curvature \cite[Ch.\,XII, Theorem 1.2]{L}. 
Using a background Riemannian metric on $\E$, we identify an inner product
with a symmetric linear operator. For each $(x,t)\in X\times \M$, 
we denote the space of inner products on $\E_{(x,t)}$ by $\T_{(x,t)}$, and so
we obtain a bundle $\T$ over $X\times \M$ with
fiber  $\T_{(x,t)}$. We equip the fibers of $\T$ with the Riemannian 
metric $d_\T$ as  above.  A   measurable (continuous)
Riemannian metric on $\E$ is a  measurable (continuous) section of $\T$. 
A  metric $\tau$ is called {\em bounded} 
if the distance between $\tau_{(x,t)}$ and $\tilde \tau_{(x,t)}$ is uniformly 
bounded on $X\times \M$ for some continuous metric $\tilde \tau$ 
on $\E$. 
The push forward of an inner product
$\tau_{(x,t)}$ on $\E_{(x,t)}$ to $\E_{F(x,t)}$ by the linear cocycle $\D$ is given by
$$
\left( (\D_{(x,t)})_* (\tau_{(x,t)}) \right) (v_1, v_2)= 
\tau_{(x,t)} \left (\D_{(x,t)}^{\,-1} (v_1),\,\D_{(x,t)}^{\,-1}(v_2) \right) \quad\text{for } v_1, v_2 \in \E_{F(x,t)}.
$$
We say that a metric  $\tau$ is $\D${\em -invariant}\,
if $\D_\ast(\tau) = \tau$.
\vskip.1cm

First, Section 4.1 of \cite{S19} gives an everywhere defined  bounded measurable 
section  $\tau$ of $\V$
invariant under the cocycle $\D$. Then Proposition 4.2 in \cite{S19} yields
that for each $x \in X$ the metric $\tau_x$ is $\gamma$-H\"older continuous on $\M$, 
more precisely, there exists a constant $c_3$ such that 
\begin{equation}\label{tau H}
d_{\T} (\tau_{(x, t)}, \tau_{(x, t')})\le c_3\, d_\M (t, t')^\gamma
\quad\text{for all } x\in X \text{ and } t, t' \in \M. 
\end{equation}

We  take any $0<\alpha<\gamma$. 
By Proposition \ref{holonomies},
boundedness of the set of values of $\A$ in Diff$^{\,1+\gamma}(\M)$ gives 
existence and regularity of the stable and unstable holonomies $H^s=H^{\A,s}$ and $H^u=H^{\A,u}$
in Diff$^{\,1+\a}(\M)$. By property (H3$^r$) of holonomies, for all  $x\in X$  and $y\in W^{s/u}_{\text{loc}}(x),$
\begin{equation}\label{h}
d_{C^r}(H^{s/u}_{\,x,y},\,\Id) \leq c_1\,d_X (x,y)^{\beta \rho}, \quad\text{where }\rho=\gamma-\alpha.
\end{equation}

We define the stable sets $\tilde W^s$ for the map $F(x,t)$ of  $X\times \M$
using the  stable holonomies. 
For any $(x,t)\in X\times \M$,
$$
  \tilde W^s(x,t)=\{(y,t')\in X\times \M: \; y\in W^s(x), \;\, t'= H^{s}_{x,y}(t)\}.
$$
These sets satisfy the contraction property
$d_{X\times \M} (F^n(x,t), F^n(y,t')) \to 0$  as $n\to \infty$
for any $(x,t)\in X\times \M$ and  $(y,t')\in \tilde W^s(x,t)$. 
The unstable sets $\tilde W^u$ are defined similarly using the  unstable holonomies.
\vskip.1cm

The next proposition establishes {\em essential invariance} of  $\tau$ under the derivatives of $H^s$ 
along the stable sets in $X\times \M$. Similar invariance holds for the
unstable holonomies. 

\begin{proposition} \cite[Proposition 4.4]{S19}\label{tau H invariant} 
Let $\nu$ be an ergodic $F$-invariant  measure on $X\times \M$.
If $\tau$ is a $\nu$-measurable  $\D$-invariant metric on $\V$, then  
 there exists an  $F$-invariant  
set $E\subset X\times \M$  with $\nu(E)=1$ such that 
 $$
\tau(y,t') = (D_tH^s_{x,y})_* (\tau(x,t))  \quad \text{for all }(x,t), (y,t') \in E 
\;\text{ with }(y,t') \in \tw^s_{loc}(x,t).
$$ 

\end{proposition}

In the {\it hyperbolic case}, we denote the {\it measure of maximal entropy}  on $X$ by $\mu$,
and in the {\it partially hyperbolic case} we denote the {\it invariant volume} on $X$ by $\mu$.
Let $m_x$ be  the normalized volume induced by the metric $\tau$ along the fiber  $\M_x$.
We define a measure $\hat \mu$ on $X\times \M$ by $\hat \mu=\int m_x \, d\mu(x)$.
This measure is $F$-invariant, but not necessarily ergodic.
Applying Proposition \ref{tau H invariant} to its ergodic components,
we obtain 

\begin{corollary} \cite[Corollary 4.5]{S19} \label{z'}
There exists a set $\hat G\subset X \times \M$ with 
$\hat \mu(\hat G)=1$ such that 
$\tau$ on $\hat G$ is invariant under the holonomies, that is, 
$$
\tau(y,t') = (D_tH^s_{x,y})_* (\tau(y,t))\, 
\text{ for all }(x,t)\in \hat G\, \text{ and all } (y,t') \in \hat G \cap \tw^s_{loc}(x,t). 
$$
\end{corollary} 

 The next proposition gives
 $\mu$-essential  invariance of $\tau_x$, as a Riemannian metric on 
the whole fiber $\M_x=\M$, under the stable and  unstable holonomies of $\A$.

\begin{proposition} \cite[Proposition 4.6]{S19} \label{suInv}
There exists a set $G\subset X$ with $\mu(G)=1$ such that 
or any $x,y,y' \in G$  with $y \in W^s_{loc}(x)$ and  $y' \in W^u_{loc}(x)$,
the diffeomorphisms
$$
H^s_{x,y}: (\M, \tau_x) \to (\M ,\tau_y) \;\text{ and }\;
H^u_{x,y'}: (\M, \tau_x) \to (\M ,\tau_{y'}) \;\text{ are isometries.}
$$
\end{proposition}

We denote by $\T^\a(\M)$ the space of $\a$-H\"older continuous 
 Riemannian metrics on $\M$ equipped with $C^{\a}$ distance $d_{\T \alpha}$.
The
 $\mu$-essential invariance of $\tau$ under the holonomies together with 
\eqref{h} yields $\mu$-essential H\"older continuity of
$\tau$ as a function from $X$ to $\T^\a(\M)$ along the stable and unstable leaves in $X$:

\begin{corollary} \cite[Corollary 4.7]{S19} \label{Holder}
The function $x \mapsto \tau_x$ is $\beta\rho$-H\"older continuous on $G$ 
along the stable and unstable leaves in $X$ as a function from $X$ to $\T^\a(\M)$, that is, 
\begin{equation}\label{Hold hol tau}
d_{\T \alpha}(\tau_x,\tau_y) \le  c_4\, d_X(x,y)^{\beta\rho} 
\quad \text{for all }x,y \in G\, \text{ with }y \in W^{s/u}_{loc}(x).
\end{equation}
\end{corollary}

\noindent Then in the the {\bf hyperbolic case}
H\"older continuous dependence of $\tau_x$ on $x\in X$ follows using 
local product structure of the measure of maximal entropy $\mu$ and of 
the stable and unstable manifolds,  as done in Section 4.5 of \cite{S19}.
\vskip.2cm

In the {\bf partially hyperbolic} case we use the following results from \cite{ASV}.
We formulate them using our notations.

 \vskip.2cm

\noindent {\bf  \cite{ASV} Definition 2.9.} {\it Let $(X,f)$ be a partially hyperbolic system,
and let $\n$ be a continuous fiber bundle over $X$.
 A {\em stable holonomy} on $\n$ is a family $h^s_{x,y} : \n_x\to \n_y$
 of $\gamma$-H\"older homeomorphisms, with uniform $\gamma>0$, 
 defined for all $x, y$ in the same stable leaf of $f$ and satisfying
\begin{itemize}
\item[(a)]  $h^s_{y,z} \circ h^s_{x,y} = h^s_{x,z}$ and $h^s_{x,x} = \Id$,
\vskip.1cm
\item[(b)]   the map $(x, y, \eta )\mapsto h^s_{x,y}(\eta)$  is continuous when $(x, y)$ varies 
in the set of pairs of  points in the same local stable leaf.
\end{itemize}
}
\vskip.2cm
\noindent Unstable holonomy is defined analogously, for pairs of points in the same unstable leaf.
\vskip.2cm

We consider the fiber bundle $\n$ over $X$ with fiber $\n_x= \T^0(\M)$ of continuous Riemannian 
metrics on $\M$ and  the maps  $h^s_{x,y}$  induced by $H^s_{x,y}$
on these metrics.  Property (a) in the definition above  holds by (H1) in Proposition \ref{holonomies}.
By (H4$^r$), the diffeomorphisms $H^s_{x,y}$ depend continuously  in Diff$^{\,1+\alpha}(\M)$ 
on $(x,y)$, where $x\in X$ and $y\in W^s_{loc}(x)$, which yields property (b). 
By the continuity, $H^s_{x,y}$  are uniformly bounded  in Diff$^{\,1+\alpha}(\M)$ over $(x,y)$ 
as above. It follows that the maps $h^s_{x,y}$ are H\"older homeomorphisms of $\n$
with uniform H\"older exponent and constant.

 \vskip.2cm
 
\noindent {\bf  \cite{ASV} Definition 2.10.} 
{\it A measurable section $\Psi : X\to \n$ of the fiber bundle $\n$
 is called $s$-invariant if $h^s_{x,y} (\Psi(x))= \Psi(y)$
 for every $x, y$ in the same stable leaf
and {\em essentially $s$-invariant} if this relation holds restricted to some full measure subset. 
The definition of $u$-invariance is analogous. Finally, $\Psi$ is {\em bi-invariant} if it is both 
$s$-invariant and $u$-invariant, and it is {\em bi-essentially invariant} if it is both essentially $s$-invariant 
and essentially $u$-invariant.}

\vskip.2cm
 A set in $X$ is called {\em bi-saturated} if it consists of full stable and unstable leaves.
  \vskip.2cm

\noindent {\bf  \cite{ASV} Theorem D.} {\it Let $f:X\to X$ be a $C^2$ partially hyperbolic 
center bunched  diffeomorphism preserving a volume $\mu$, and let $\n$ be 
a continuous fiber bundle with stable and unstable holonomies and with refinable fiber. 
Then, 
\begin{itemize}
\item[(a)] for every bi-essentially invariant section $\Psi : X\to \n$, there exists a bi-saturated 
set $X_\Psi$ with full measure, and a bi-invariant section $\tilde \Psi : X_\Psi \to \n$ that 
coincides with $\Psi$ at $\mu$ almost every point.
\vskip.1cm
\item[(b)]  if $f$ is  accessible then $X_\Psi=X$ and $\tilde \Psi$ is continuous.
\end{itemize}
}
\vskip.2cm

By remark after Definition 2.10 in \cite{ASV} every Hausdorff topological space with a countable 
basis of topology is refinable. Since the space $\T^0(\M)$ is separable, the fiber $\n_x= \T^0(\M)$ 
 of $\n$ is refinable. The section $\Psi(x)=\tau_x$ of $\n$  is bi-essentially invariant 
by Proposition \ref{suInv}. Applying Theorem D,  we and conclude that, up to modification on a set of measure zero, $\tau$ is continuous as a function from $X$ to $\T^0 (\M)$ and 
bi-invariant on  $X$. The latter implies  that Corollary \ref{Holder} hold on $G=X$,
and part (d) follows.
\vskip.1cm

It remains to show that $\tau$ is continuous  as a function from $X$ to $\T^\a (\M)$.
Since  $\tau$ is a continuous function from a compact set $X$ to $\T^0 (\M)$, 
it is bounded in $\T^0 (\M)$. The estimate  \eqref{tau H} implies that  
$$
  \sup\, \{\, d_{\T} (\tau_{(x, t)}, \tau_{(x, t')})\cdot (d_\M (t, t'))^{-\gamma}:  
  \;x\in X \text{ and }t\ne t' \in \M \} \,\le\, c_3.
$$
Hence the set $\{ \tau_x:\; x\in X\}$ is bounded in $\T^\gamma(\M)$.
Since $\a<\gamma$, the embedding of $\T^\gamma (\M)$ into $\T^\a(\M)$ is compact,
see for example \cite[Proposition 3.3]{dlLO},
and hence  the set $\{ \tau_x\}$ has compact closure in $\T^\a(\M)$.
It follows that $\tau:X\to \T^\a (\M)$ is continuous.
Indeed, suppose that $x_n\to x$ in $X$, but $\tau_{x_n} \not\to \tau_x$ in $\T^\a (\M)$.
Then $(\tau_{x_n})$ has a subsequence converging to some $\hat \tau \ne \tau_x$ in $\T^\a (\M)$
and hence in $\T^0(\M)$,  which contradicts the convergence of $(\tau_{x_n})$ to $\tau_x$ in $\T^0 (\M)$.
Thus $\tau_{x_n} \to \tau_x$ in $\T^\a (\M)$.

 \vskip.1cm
This completes the proof for the partially hyperbolic case.


\section{Proofs of Proposition \ref{prop of inter conj}, Theorem \ref{partially hyperbolic} and Corollary \ref{hyperbolic}} \label{5}


\subsection{Proof of Proposition \ref{prop of inter conj}}
The last part of the proposition will be used in the proof of Theorem~\ref{partially hyperbolic}.
\vskip.2cm
\noindent {\bf (a)} and {\bf (b)} follow immediately from the Definitions 1.5 and  \eqref{weight}.
\vskip.2cm

\noindent {\bf (c)} We fix $x\in X$. Since $f$ is accessible, for every $y\in X$ there is an $su$-path 
$P=P_{x,y}$ from $x$ to $y$. Then by (b) we have 
\begin{equation}\label{P y}
\P_y=\H^{\A,P}_{x,y} \circ \P_x\circ  (\H^{\B,P}_{x,y})^{-1}.
\end{equation}
\vskip.2cm

\noindent {\bf (d,e)} 
Suppose that  $H^{\A,\, s/u}$ and $H^{\B,\, s/u}$ are in  $\drm$, where $r=0$ or $r\ge 1$. 
Then so are $\H^{\A,P}$ and $\H^{\B,P}$ for any $su$-path $P$, and  accessibility 
together with \eqref{P y} imply that if $\P_{x_0}\in \drm$, then $\P(y) \in \drm$ for all $y\in X$.
\vskip.1cm

Now for $r\ge 1$ we show that the function $\P :X \to \drm$ is bounded. 
The holonomies $H^\A$ and $H^\B$ uniformly bounded in $\drm$ over all 
$x\in X$ and $y\in W^s_{loc}(x)$
by the continuity property (H4), and we set
$$
K_H=\sup\,\{ \,|H_{x,y}^{\mathcal{D},s}|_{C^r}: \; x\in X,\; y\in W^s_{loc}(x),\; \mathcal{D}=\A,\B\}.
$$
For any $x\in X$ and $y\in W^s_{loc}(x)$
we have $\P_y=H^{\A,\, s/u}_{x,y} \circ \P_x\circ  H^{\B,\,s/u}_{y,b}$. For $r\ge 1$, 
we use Lemma \ref{dlLO lemma} twice to obtain the estimate
$$
\begin{aligned}
 \|\P_y\|_{C^r}&= \|H^{\A}_{x,y} \circ \P_x\circ  H^{\B}_{y,x}\|_{C^r}\le 
 M_r^2\, \|H^{\A}_{x,y}\|_{C^r}  (1+  \|\P_x\|_{C^r})^r  (1+  \|H^{\B}_{y,x}\|_{C^r})^r\le \\
&\le  M_r^2\, K_H\cdot  (1+  \|\P_x\|_{C^r})^r (1+K_H)^r =
 K'\cdot (1+  \|\P_x\|_{C^r})^r
 \end{aligned}
$$
 Since $f$ is accessible,
 there exist constants $L$ and $K$ such that 
for any  $x,y\in X$ there exists an $su$-path from $x$ to $y$ with at most $L$ subpaths 
of length at most $K$, each lying in a single leaf of $W^s$ or $W^u$ \cite[Lemma 4.5]{W}.
For any $y\in X$, we consider such a path from $x_0$ to $y$
 and, starting with $\P_{x_0}$, apply the above estimate a bounded number of times. 
 Thus we obtain
$$
\|\P_y\|_{C^r}\le K''=K''(K_H,  \|\P_{x_0}\|_{C^r},K,L,r)\; \text{  for all $y\in X$ }.
$$

Now we establish continuity of $\P$. 
For $r=0$ or $r\ge 1$,  we let $\ell=\lfloor r \rfloor$ be the integer part of $r$,
and consider the fiber bundle $\n$ over $X$ with fiber $\dlm$. For any $x\in X$, $y\in W^s(x)$, 
and $y'\in W^u(x)$  we define maps $h^s_{x,y}: \n_x \to \n_y$ and $h^u_{x,y'}: \n_x \to \n_{y'}$ 
by 
$$
h^s_{x,y} (g) =  H^{\A,s}_{x,y} \circ g \circ H^{\B,s}_{y,x}
 \quad \text{and} \quad h^u_{x,y'} (g) =H^{\A,u}_{x,y'} \circ g\circ H^{\B,u}_{y',x}.
$$
 Since $ \P$ intertwines 
$H^\A$ and $H^\B$ we have $\P(y) =h^s_{x,y}(\P_x)$ and $\P(y') =h^s_{x,y'}(\P_x)$, that is, $\P$ is invariant under $h^s$ and $h^u$. 
It is well known that for any $\ell \in \N\cup\{0\}$,\, $\dlm$ is a topological group  
and so the map $(g,h)\mapsto g\circ h$ from $\dlm \times \dlm$ to $\dlm$ is continuous.
Also, by property (H4) of the holonomies, the maps $(x,y) \mapsto H^{\A/\B,\, s/u}_{x,y}$ 
into $\dlm$ are continuous on the set of pairs $(x,y)$ where $y\in W^s_{loc}(x)$, and it follows 
that the maps $(x,y,g) \mapsto h^{s/u}_{x,y}(g)$ into $\dlm$  are continuous. 
Therefore, by its invariance, $\P$ a bi-continuous section of $\n$ in the sense 
of the definition below.

\vskip.2cm
\noindent {\bf  \cite{ASV} Definition 2.12.}
{\it A measurable section $\Psi : X\to \n$ of a continuous fiber bundle $\n$ is $s$-continuous 
if the map $(x, y, \Psi(x) )\mapsto \Psi(y)$ is continuous on the set of pairs of points $(x,y)$ 
in the same local stable leaf. The  $u$-continuity is defined similarly using unstable leaves. 
Finally, $\Psi$ is bi-continuous if it is both $s$-continuous and $u$-continuous.}
\vskip.2cm
We apply the theorem below to conclude that $ \P:X\to \dlm$ is continuous,
completing the proof for the case of an integer $r=\ell$.
\vskip.2cm
\noindent {\bf  \cite{ASV} Theorem E.} {\it Let $X\to X$ be a $C^1$ partially hyperbolic 
accessible diffeomorphism and $\n$ be a continuous fiber bundle. 
Then  every bi-continuous section $ \Psi : X\to \n$ is continuous on $X$.}
\vskip.2cm
 The argument above does 
 not apply with a non-integer $r$ in place of its integer part $\ell$ since the composition 
of $C^r$ maps does not depend continuously on the terms in $C^r$ distance in general, 
see \cite[Example 6.4]{dlLO}.

Finally, suppose that $r$ is not an integer.  As we showed above, $\P:X\to \drm$  is bounded
 and $\P:X\to \dlm$ is continuous. We take $p$ such that $\ell<p<r$. Since the embedding 
 of $\drm$ into $\dpm$ is compact, it follows as in the end of the proof of Theorem \ref{isometric} 
 that $\P:X\to \dpm$ is continuous.

\vskip.2cm
 

\subsection{Proof of Theorem \ref{partially hyperbolic}}
Part (a) of Theorem \ref{partially hyperbolic} follows from 
Propositions \ref{intertwines} and \ref{intertwines to cont} below, 
and Proposition \ref{prop of inter conj}(e). In Proposition \ref{intertwines} 
we prove that a measurable conjugacy intertwines the holonomies of 
$\A$ and $\B$ on a set of full measure. Then in Proposition \ref{intertwines to cont}
we show that it  coincides on a set of full measure with  a continuous conjugacy 
which intertwines the holonomies on $X$. Finally, we apply the 
Proposition \ref{prop of inter conj}(e) to obtain the regularity of the conjugacy.

\begin{proposition}\label{intertwines}
Let $(X,f)$ be either a partially hyperbolic diffeomorphism 
or a hyperbolic system, and 
let $\mu$ be an ergodic $f$-invariant  measure. 
Let $\A$ and $\B$ be $\dzm$-valued  cocycles over $(X,f)$.
Suppose that the set $\{\B^x_n:\; x\in X, \; n\in \Z\}$
has compact closure in $\dzm$ and  that 
for any $x\in X$ and $y\in W^s(x)$, 
$$
H^{\A,s}_{x,y}=\lim_{n\to+\infty} (\A^n_y)^{-1} \circ \A^n_x \;\text{ and }\;
H^{\B,s}_{x,y}=\lim_{n\to+\infty} (\A^n_y)^{-1} \circ \A^n_x \;\text{  exist in $\dzm$.}
$$
Let $\P:X\to \dzm$ be a $\mu$-measurable conjugacy between $\A$ and $\B$.
Then $\P$ intertwines the stable holonomies $H^{\A,s}$ and $H^{\B,s}$ of $\A$ and $\B$
on a set of full measure.

A similar statement holds for the unstable holonomies.
\end{proposition}

We note that continuity of the map $(x,y)\mapsto H_{x,y}^{\A,s}$ is not assumed in this proposition.

\begin{proof}
We will give a proof for the stable holonomies.
We will show that for all $x$ and $y\in W^s(x)$ in a set of full measure,
\begin{equation}\label{int meas}
\P_y^{-1} \circ H_{x,y}^{\A,s} \circ \P_x = H_{x,y}^{\B,s} .
\end{equation}

Since the map $\P$ is $\mu$-measurable and the space $\dzm$ is separable, 
by Lusin's theorem there exists a compact set $S\subset X$ with $\mu(S)>1/2$ 
such that $\P:X\to \dzm$ is uniformly continuous on $S$. 

\vskip.1cm

Let $Y$ be the set of points in $X$ for which the frequency of 
visiting the set  $S$ equals $\mu(S)>1/2$. By Birkhoff Ergodic Theorem,
$\mu(Y)=1$. 
Let $x$ and  $y\in W^s (x)$ be in $Y$. Then there exists a sequence $\{n_i\}$
such that $f^{n_i}x$ and $f^{n_i}y$ are in $S$ for all $i$. 
We denote $x_n=f^n x$ and $y_n=f^n y$.
Since $d_X(x_{n_i} , y_{n_i}) \to 0$ as $i\to\infty$ and $\P$ is uniformly continuous on $S$,
$$
  d_{C^0}(\P_{x_{n_i}},  \P_{y_{n_i}}) \to 0 \quad \text{ as }i\to\infty.
$$
For $d_0$ as in \eqref{d0} it follows that 
\begin{equation}\label{P1}
d_0  (\P_{y_{n_i}}^{-1} \circ \P_{x_{n_i}}, \, \Id) \to 0\;\text{ as }i\to\infty.
\end{equation}

Now we establish \eqref{int meas}.
Since   $\A_x^{n_i}=\P_{x_{n_i}} \circ \B_x^{n_i} \circ  \P_x^{-1}$,  we have
\begin{equation}\label{AA}
\P_y^{-1} \circ (\A^{n_i}_y)^{-1} \circ  \A^{n_i}_x \circ  \P_x
\,=\, (\B^{n_i}_y)^{-1} \circ \P_{y_{n_i}}^{-1} \circ  \P_{x_{n_i}} \circ \B^{n_i}_x .
\end{equation}
We show that the left-hand side converges to $\P_y^{-1} \circ H_{x,y}^{\A,s}  \circ  \P_x$ and 
the right-hand side converges to  $H_{x,y}^{\B,s}$ in $d_0$. 

For a   homeomorphism  $g$ of $\M$ and $\delta>0$ we define
$$
\omega_g(\delta)=\sup \, \{\, d_\M \left(g^\ast(t),g^\ast(t') \right):\; \ast = 1,-1, \;\, t, t'\in \M\text{ and }d(t,t')\le \delta\,\}.
$$
Since $g$ is uniformly continuous on $\M$, $\,\omega_g(\delta)\to 0$ as $\delta\to 0$.

Since $\{ \B_x^n:\; x\in X, \; n\in \Z \}$ has compact closure in $\dzm$,
 the family $\{ \B_x^n \}$ is uniformly equicontinuous. It follows that 
$$
\omega_\B(\delta)=\sup\, \{\, \omega_{\B_x^n}(\delta):\; x\in X, \; n\in \Z\,\} \to 0 \,\text{ as }\delta \to 0.
$$
We observe that if $g, h_n,k \in \dzm$ and $h_n\to h$ in $\dzm$,
then
$$
  d_0 (g\circ h_n \circ k, \, g\circ h \circ k) =  d_0 (g\circ h_n, \, g\circ h)
  \le  \omega_g ( d_0 (h_n, h)) \to 0.
$$
 Since $(\A^n_y)^{-1} \circ  \A^n_x \to H_{x,y}^{\A,s} $ in $\dzm$, it follows that 
$$
d_0\left(\P_y^{-1} \circ (\A^n_y)^{-1} \circ  \A^n_x \circ  \P_x,\,\, 
\P_y^{-1} \circ H_{x,y}^{\A,s}  \circ  \P_x \right)  \to 0 \quad\text{as $n\to\infty$.}
$$
Denoting $g_{n_i}= (\B^{n_i}_y)^{-1}$,\, $h_{n_i}=\P(y_{n_i})^{-1} \circ  \P(x_{n_i})$, and 
$k _{n_i}= \B^{n_i}_x$,  we estimate
$$
  d_0(g_{n_i}\circ h_{n_i}\circ  k_{n_i},\, g_{n_i} \circ \Id \circ k_{n_i}) =
  d_0(g_{n_i}\circ h_{n_i},\, g_{n_i} \circ \Id) \le
   \omega_\B \,(d_0( h_{n_i}, \Id ) ).
$$ 
Since $d_{C^0}( h_{n_i}, \Id )  \to 0$ by \eqref{P1}, we obtain 
$$
d_0\left( (\B^{n_i}_y)^{-1} \circ \P_{y_{n_i}}^{-1} \circ  \P_{x_{n_i}} \circ \B^{n_i}_x, \,\,
 (\B^{n_i}_y)^{-1}\circ\B^{n_i}_x \right) \to 0.
$$
Finally, as $d_0 ( (\B^n_y)^{-1} \circ\B^n_x, \,\, H^{\B,s}_{x,y} )\to 0$, 
$$
d_0 ((\B^{n_i}_y)^{-1} \circ \P_{y_{n_i}}^{-1} \circ  \P_{x_{n_i}} \circ \B^{n_i}_x, \,\, H_{x,y}^{\B,s})\to 0.
$$
Therefore, \eqref{AA} together with  the above estimates imply that 
$$
\P_y^{-1} \circ H_{x,y}^{\A,s} \circ \P_x = H_{x,y}^{\B,s}, 
\;\text{ equivalently, }\;  H_{x,y}^{\A,s}=\P_y\circ H_{x,y}^{\B,s} \circ \P_x^{-1}, 
$$ 
and we conclude that  $\P$ intertwines the stable 
holonomies of $\A$ and $\B$ on the set $Y$.
\end{proof}


\begin{proposition}\label{intertwines to cont}
Let $f:X\to X$ be an accessible  center bunched $C^2$ partially hyperbolic diffeomorphism 
preserving a volume $\mu$. Let $\A$ and $\B$ be $\dqm$-valued  cocycles over $(X,f)$
with holonomies in $\dzm$.

Let $\P:X\to \dzm$ be a $\mu$-measurable conjugacy between $\A$ and $\B$
which intertwines their holonomies on a set $Y\subseteq X$ of full measure. 
Then $\P$ coincides on a set of full measure with  a continuous conjugacy $\tilde \P:X\to \dzm$
which intertwines the holonomies of $\A$ and $\B$ on $X$.
\end{proposition}

\begin{proof}
For every $x\in Y$, $y\in W^s(x)\cap Y$ and $y'\in W^u(x)\cap Y$ we have
\begin{equation} \label{Phi}
\P_y=H^{\A,s}_{x,y} \circ \P_x\circ  (H^{\B,s}_{x,y})^{-1}= H^{\A,s}_{x,y} \circ \P_x\circ H^{\B,s}_{y,x}
\quad \text{and} \quad \P_{y'}= H^{\A,u}_{x,y'} \circ \P_x\circ H^{\B,u}_{y',x}.
\end{equation}

Now we apply  \cite[Theorem D]{ASV} stated in the last part of the proof of Theorem~\ref{isometric}.
We consider the fiber bundle $\n$ over $X$ with fiber $\dzm$. Since the space $\dzm$ is separable,
the fiber bundle is refinable.
As in the proof of Proposition~\ref{prop of inter conj}, 
for  any $x\in X$, $y\in W^s(x)$, and $y'\in W^u(x)$  
we consider  the maps $h^s_{x,y}: \n_x \to \n_y$ and $h^u_{x,y'}: \n_x \to \n_{y'}\,$ given by 
$$
h^s_{x,y} (g) =  H^{\A,s}_{x,y} \circ g \circ H^{\B,s}_{y,x}
 \quad \text{and} \quad h^u_{x,y'} (g) =H^{\A,u}_{x,y'} \circ g\circ H^{\B,u}_{y',x}.
$$
The family $\{ h^s_{x,y} \}$ is a stable holonomy in the sense of \cite{ASV} Definition~2.9.
Indeed, by property (H1) of $H^{\A,s}$ and $H^{\B,s}$ we have $h^s_{x,x}=\Id $, and for
any $y,z\in W^s(x)$, 
$$
(h^s_{y,z} \circ h^s_{x,y})(g) =
H^{\A,s}_{y,z} \circ H^{\A,s}_{x,y} \circ g \circ H^{\B,s}_{y,x} \circ H^{\B,s}_{z,y}=
H^{\A,s}_{x,z} \circ g \circ H^{\B,s}_{z,x} = h^s_{x,z}(g).
$$
Since the map $(g,h)\mapsto g\circ h$ from $\dzm \times \dzm$ to $\dzm$ is continuous, 
and  $H^{\A,s}_{x,y}$ and $H^{\A,s}_{x,y}$ depend continuously on $(x,y)$ with $y\in W^s_{loc}(x)$, 
 it follows that the map $(x,y,g) \mapsto h^s_{x,y}(g)$
is continuous. Also, by property (H5$^0$) the maps $H^{\A,s}_{x,y}$ and $H^{\B,s}_{x,y}$ 
are H\"older with a uniform constant $K$ and exponent $\gamma$. Then for any $g,g'\in \dzm$ we have
$$
  d_0 (H^{\A,s}_{x,y} \circ g \circ H^{\B,s}_{y,x}, \,H^{\A,s}_{x,y} \circ g' \circ H^{\B,s}_{y,x}) =  
  d_0 (H^{\A,s}_{x,y} \circ g, \, H^{\A,s}_{x,y} \circ g') \le K\,  d_0 (g, g')^\gamma.
$$
A similar estimate holds for the inverses, and we obtain 
$$
d_{C^0} (h^s_{x,y} (g),\, h^s_{x,y} (g) ) \le K\, d_{C^0} (g, g')^\gamma.
$$
Thus the homeomorphisms
$h^s_{x,y}$ are also H\"older with uniform constant and exponent.
Similarly, we see that $\{ h^u_{x,y'} \}$ is an unstable holonomy. 

Since  \eqref{Phi} can be restated as 
$\P_y=h^s_{x,y}(\P_x)$ and $ \P_{y'}=h^u_{x,y'}(\P_x),$
the conjugacy $\P$ is a bi-essentially invariant section of $\n$ as in Definition 2.10 in \cite{ASV}. 
Then Theorem D in \cite{ASV} yields that
 $\P$ coincides on a set of full measure with 
a continuous conjugacy $\tilde \P:X\to \dzm$ that  is invariant under $h^s$ and $h^u$.
The latter  means that is $\tilde \P$ intertwines the holonomies of $\A$ and $\B$ on $X$. 
\end{proof}

For $r=0$, the two propositions above yield part (a) of the theorem.
To complete the proof for $r\ge 1$, we apply the last part of Proposition \ref{prop of inter conj}
to the continuous conjugacy $\tilde \P:X\to \dzm$ which intertwines the holonomies.
This completes the proof of  (a).

\vskip.3cm

\noindent {\bf (b)}$\;$
Now we assume that $r>1$ and the holonomies of $\A$ and $\B$ satisfy the H\"older condition \eqref{beta'}.
We write $\P$ for $\tilde \P$ to simplify the notations.
We give a proof of H\"older continuity of $\P$ along $W^s$,  the argument for $W^u$ is
 similar, using the unstable holonomies.
 The stable holonomies of $\A$ and $\B$ are uniformly bounded in $\drm$ 
over all $x\in X$ and $y\in W^s_{loc}(x)$ by (H4$^r$), 
 and we showed already that $\P:X\to \drm$ is bounded.
So we  set
$$
\begin{aligned}
& K_H=\sup\,\{ \,|H_{x,y}^{\mathcal{D},s}|_{C^r}: \; x\in X,\; y\in W^s_{loc}(x),\; \mathcal{D}=\A,\B\},\\
&K_\P=\sup\, \{ |\P_x|_{C^r}: x\in X\}.
\end{aligned}
$$
For any $x\in X$ and $y\in W^s_{loc}(x)$ we use Lemma \ref{dlLO lemma} to estimate
$$
\begin{aligned}
  & \|H_{x,y}^{\A,s} \circ \P_x\|_{C^{r}} 
  \le M_r \,\| H_{x,y}^{\A,s}\|_{C^{r}}\,(1+\|\P_x\|_{C^{r}})^{r} 
  \le M_{r} K_H (1+K_\P)^r = : K_1, \; \text{ and}\\
  & \|\P_x^{-1}\circ (H_{x,y}^{\A,s})^{-1}\|_{C^{r}}
   \le M_{r} \,\| \P_x^{-1}\|_{C^{r}}\,(1+\|(H_{x,y}^{\A,s})^{-1}\|_{C^{r}})^{r} 
   \le M_r K_\P(1+K_H)^{r} =: K_2. 
   \end{aligned}
$$
If $r$ is an integer, we take $r-1\le p'<r$,  if  $r$ is not an integer, we take $\lfloor r\rfloor \le p'<r$, 
and we set $\rho =r-p'$. 
In the estimates below, we use  Lemma \ref{distance} with $q=r$, $r=p'$, and 
either $g = \Id$ or $\tilde  g=\Id$:
$$
\begin{aligned}
&d_{C^{p'}} (h_1\circ \tg,\, h_2 \circ \tg) 
\le M \left( 2^{p'}\| \tg^{-1}\|_{C^r}   + (1+\| \tg\|_{C^{p'}})^{p'}\right)
 \cdot d_{C^{p'}}(h_1, h_2)^\rho,\\
&d_{C^{p'}} (g\circ h_1,\, g\circ h_2 ) 
\le M \left(  2^{p'} \| g\|_{C^r} +(1+ \| g^{-1}\|_{C^{p'}})^{p'} \right)
 \cdot d_{C^{p'}}(h_1, h_2)^\rho.
\end{aligned}
$$
Since $\P$ intertwines the holonomies,
$\P_y=  H_{x,y}^{\A,s} \circ \P_x\circ H_{y,x}^{\B,s},$ and we estimate 
$$
\begin{aligned}
 &d_{C^{p'}}(\P_x,\, \P_y)= d_{C^{p'}}\left( \P_x,\, H_{x,y}^{\A,s} \circ \P_x\circ H_{y,x}^{\B,s} \right) \le \\
 &\le d_{C^{p'}}(\Id\circ \P_x,\,H_{x,y}^{\A,s} \circ \P_x)+ d_{C^{p'}}\left( (H_{x,y}^{\A,s} \circ \P_x)\circ\Id,\, (H_{x,y}^{\A,s} \circ \P_x)\circ H_{y,x}^{\B,s} \right) \le\\
&\le  M \cdot \left( 2^{p'}\, \| \P_x^{-1}\|_{C^r} +(1+ \| \P_x\|_{C^{p'}})^{p'} \right)
\cdot \left( d_{C^{p'}}(\Id,\,H_{x,y}^{\A,s})\right)^{\rho} + \\
&+M \cdot \left( 2^{p'}\, \| H_{x,y}^{\A,s} \circ \P_x\|_{C^r}  +(1+ \| (H_{x,y}^{\A,s} \circ \P_x)^{-1}\|_{C^{p'}})^{p'} \right)
 \cdot \left( d_{C^{p'}}(\Id,\,H_{y,x}^{\B,s})\right) ^{\rho} \le\\
&\le  K_3\cdot \left( d_{C^{r}}(\Id,\,H_{y,x}^{\A,s})\right) ^{\rho} + 
K_4\cdot \left( d_{C^{r}}(\Id,\,H_{y,x}^{\B,s})\right) ^{\rho} \le
K_5 \cdot    d_X(x,y)^{\beta'\rho}.
\end{aligned}
$$
In each of the two applications of Lemma \ref{distance} above, the assumptions of the lemma
 are satisfied. Indeed, since $h_1=\Id$ and $h_2=H_{x,y}^{\A/\B,s}$ we have 
 $|h_1|_{C^r}, |h_2|_{C^r} \le K_H $, and by the property \eqref{beta'}
we have $d_{C^{p'}}(h_1, h_2)\le \delta_0 |h_1|_{C^{p'}}^{-1}$ for all sufficiently close 
$x$ and $y\in W^s_{loc}(x)$. Also, the expression in \eqref{delta 0} is uniformly bounded below
by some $c''>0$, and so the second assumption also holds provided that $d_X(x,y)$ is small enough. 
\vskip.1cm

We conclude that for any sufficiently close $x\in X$ and $y\in W^s_{loc}(x)$,
\begin{equation}\label{Holder P}
d_{C^{p'}}(\P_x,\, \P_y) \le K_5 \cdot    d_X(x,y)^{\beta'\rho}.
\end{equation}
The same estimate holds any sufficiently close $x\in X$ and $y\in W^u_{loc}(x)$.
This concludes the proof of Theorem \ref{partially hyperbolic}.


\vskip.3cm

\noindent{\bf Proof of Corollary \ref{hyperbolic}}.
We already established H\"older continuity of $\P$ along $W^s$ and $W^u$.
Let $x\in X$, and let $z$ be sufficiently close to $x$ so that the intersection 
$W^s_{loc}(x)\cap W^u_{loc}(z)$ consists of a single point, which we denote by $y$.
Then by \eqref{Holder P},
$$
  d_{C^{p'}}(\P_x,\, \P_y) \le K_5 \cdot    d_X(x,y)^{\beta'\rho}\quad\text{and}\quad
  d_{C^{p'}}(\P_y,\, \P_z) \le K_5 \cdot    d_X(y,z)^{\beta'\rho},
$$
and it follows that   $d_{C^{p'}}(\P_x,\, \P_z) \le K_6 \cdot    d_X(x,z)^{\beta'\rho}$.


\section{Proofs of Proposition \ref{C0 int}, Theorem \ref{sufficient}
and Corollary \ref{cohomologous to constant}} \label{6}

\subsection{Proof of Proposition \ref{C0 int}} $\;$
Both  cocycles have stable holonomies in $\dzm$ by Proposition \ref{C0 hol}.
\vskip.1cm

In the proof we use only H\"older continuity of $\P$ along $W^s$, specifically,  
that there exists a constant $K_1$ such that 
$$
d_{C^0} (\P_x, \P_y) \le K_1 \, d_X(x,y)^\b 
\;\text{ for all $x\in X$ and $y\in W^s_{loc}(x)$.}
$$
By the invariance property (H2) of holonomies, it suffices to prove the intertwining for $y\in W^s_{loc}(x)$.
We fix  $x\in X$ and  $y\in W^s_{loc}(x)$. As in \eqref{dAA} we obtain that for all $n\in \N$, 
$$
 d_0  (\P_{y_n}^{-1} \circ \P_{x_n}, \, \Id) \le K_1 \, d_X (x,y)^\b \cdot \la^{n\b}.
$$
Since   $\A_x^n=\P_{fx_n} \circ \B_x^n \circ  \P_x^{-1}$,  we have
\begin{equation}\label{AAA}
\P_y^{-1} \circ (\A^n_y)^{-1} \circ  \A^n_x \circ  \P_x
\,=\, (\B^n_y)^{-1} \circ \P_{y_n}^{-1} \circ  \P_{x_n} \circ \B^n_x .
\end{equation}

Since $\P_y$ is a homeomorphism of a compact manifold, 
$\P_y^{-1}$ is uniformly continuous on $\M$.  
Since $(\A^n_y)^{-1} \circ  \A^n_x \to H_{x,y}^{\A,s} \,$ uniformly on $\M$, it follows that 
$$
\begin{aligned}
&d_0\left(\P_y^{-1} \circ (\A^n_y)^{-1} \circ  \A^n_x \circ  \P_x,\,\, 
\P_y^{-1} \circ H_{x,y}^{\A,s}  \circ  \P_x \right)=\\
&= d_0\left(\P_y^{-1} \circ ((\A^n_y)^{-1} \circ  \A^n_x),\,\, 
\P_y^{-1} \circ H_{x,y}^{\A,s} \right)  \to 0 \quad\text{as $n\to\infty$.}
\end{aligned}
$$
Also, 
$$
\begin{aligned}
& d_0\left( (\B^n_y)^{-1} \circ \P_{y_n}^{-1} \circ  \P_{x_n} \circ \B^n_x, \,\,
 (\B^n_y)^{-1}\circ\B^n_x \right) = \\
 & =d_0\left( (\B^n_y)^{-1} \circ (\P_{y_n}^{-1} \circ  \P_{x_n}), \,\,
 (\B^n_y)^{-1} \circ \Id \right) \le  K \sigma^n\cdot d_0\left(  (\P_{y_n}^{-1} \circ  \P_{x_n}),  \Id \right) \le   \\
 &\,\,
\le  K \sigma^n \cdot K_1 \,d_X (x,y)^\b \cdot \la^{n\b} =  K K_1 \,d_X (x,y)^\b \cdot (\sigma \la^\b)^n 
\to 0 \quad\text{as $n\to\infty$}.
 \end{aligned}
$$
Since  $d_0 ( (\B^n_y)^{-1} \circ\B^n_x, \,\, H^{\B,s}_{x,y} )\to 0$, it follows that
$$
d_0 ((\B^n_y)^{-1} \circ \P_{y_n}^{-1} \circ  \P_{x_n} \circ \B^n_x, \,\, H_{x,y}^{\B,s})\to 0.
$$
Thus the left-hand side of \eqref{AAA} converges to $\P_y^{-1} \circ H_{x,y}^{\A,s}  \circ  \P_x$ and 
the right-hand side converges to  $H_{x,y}^{\B,s}$ in $C^0(\M)$,
and we conclude that $\P$ intertwines the stable holonomies of $\A$ and $\B$.


\subsection{Proof of Theorem \ref{sufficient}}
We will prove (b), and then (a) follows.
\vskip.1cm

We write $x$ in place of $x_0$.  For every $y\in X$ we define 
\begin{equation}\label{Py}
\P_y=\H^{\A,P}_{x,y} \circ \P_x\circ  (\H^{\B,P}_{x,y})^{-1},
\;\text{ where $P=P_{x,y}$ is an $su$-path from $x$ to $y$.}
\end{equation}
The   value $\P_y$ does not depend on the choice of a path from $x$ to $y$, and hence is well-defined.
Indeed, let  $(P_{x,y})^{-1}=\{y=x_k, x_{k-1}, \dots , x_1,   x_0=x\}$,  let
$\tilde P_{x,y}$ be another $su$-path $P_{x,y}$ from $x$ to $y$, and let $\tilde \P_y$ 
be the corresponding value. Then $(P_{x,y})^{-1} \tilde P_{x,y}$ is an $su$-cycle, and 
using  (b1) we obtain
$$
\begin{aligned}
 & \P_y^{-1}\circ \tilde \P_y = \H^{\B,P}_{x,y} \circ \P_x^{-1}\circ  (\H^{\A,P}_{x,y})^{-1}\circ
 \H^{\A,\tilde P}_{x,y} \circ \P_x\circ  (\H^{\B, \tilde P}_{x,y})^{-1} =\\
 &= \H^{\B,P}_{x,y} \circ \P_x^{-1}\circ \H^{\A,\,P^{-1}\tilde P}_{x} \circ \P_x\circ  (\H^{\B, \tilde P}_{x,y})^{-1}= \H^{\B,P}_{x,y} \circ \H^{\B,\,P^{-1}\tilde P}_{x} \circ  (\H^{\B, \tilde P}_{x,y})^{-1} =\Id.
 \end{aligned}
$$
Let $z\in W^{s/u}(y)$ then 
$$
\P_z=H^{\A,\, s/u}_{y,z}\circ \H^{\A,P}_{x,y} \circ \P_x\circ  (\H^{\B,P}_{x,y})^{-1} \circ (H^{\B,\,s/u}_{y,z})^{-1} =H^{\A,\, s/u}_{y,z}\circ  \P_y\circ  (H^{\B,\,s/u}_{y,z})^{-1},
$$
and so $\P$ intertwines the holonomies. Then it follows by Proposition \ref{prop of inter conj}(e)
that $\P:X\to \drm$ is bounded  
and $\P:X\to \dpm$ is continuous for $p=r$ if $r$ is an integer,
and for any $p<r$ otherwise.

\vskip.1cm 

It remains to show that $\P$ is a conjugacy, that is, itsatisfies 
$$
\A_y=\P_{fy} \circ \B_y \circ \P_y^{-1}  \quad\text{for all }y\in X.
$$
We consider a point $y\in \M$ and  an $su$-path $P=P_{x,y}$ from $x$ to $y$.
Then $f(P)$ is an  $su$-path from $fx$ to $fy$.  It follows from the definition of $\P$ that 
 for any $z,w\in \M$ and  any $su$-path 
$P_{z,w}$ from $z$ to $w$ we have
$\P_w=\H^{\A,P}_{z,w} \circ \P_z\circ  (\H^{\B,P}_{z,w})^{-1},$
and in particular for $z=fx$ and $w=fy$,
\begin{equation}\label{C1}
\P_{fy} = \H^{\A,\,f(P)}_{fx,fy} \circ \P_{fx}\circ  (\H^{\B,\,f(P)}_{fx,fy})^{-1}.
\end{equation}
By properties (H2, H2$'$) of the holonomies, for any $z\in X$, $w\in W^{s/u}(z)$,
$$
H^{\A,\,s/u}_{fz,fw} = \A_w \circ H^{\A,\,s/u}_{z,w} \circ (\A_z)^{-1},
$$
and it follows that 
\begin{equation}\label{C2}
\H^{\A,\,f(P)}_{fx,fy} = \A_y \circ \H^{\A,\,P}_{x,y} \circ (\A_x)^{-1} \,\text{ and similarly }\;
\H^{\B,\,f(P)}_{fx,fy} = \B_y \circ \H^{\B,\,P}_{x,y} \circ (\B_x)^{-1}.
\end{equation}
Since the definition of $\P_{fx}$ in (b2) is consistent with $\eqref{Py}$, by (b2) we have
\begin{equation}\label{C3}
  (\A_x)^{-1} \circ \P_{fx}\circ \B_x = \P_x.
\end{equation}
Combining \eqref{C1}, \eqref{C2}, and \eqref{C3} we obtain
$$
\begin{aligned}
\P_{fy} & =\A_y \circ \H^{\A,P}_{x,y} \circ \A_x^{-1} \circ \P_{fx}\circ
\B_x \circ (\H^{\B,P}_{x,y})^{-1} \circ \B_y^{-1} = \\
&=\A_y \circ \H^{\A,P}_{x,y}\circ \P_x \circ (\H^{\B,P}_{x,y})^{-1} 
\circ \B_y^{-1}  =\A_y \circ \P_y\circ  \B_y^{-1} .
\end{aligned}
$$


\subsection{Proof of Corollary \ref{cohomologous to constant}}
Let $\B$ be a constsant cocycle.
Then its stable and unstable  holonomies are trivial, that is, $H_{x,y}^{\B,\, s/u}=\Id$, and hence $\H^{\B,P}=\Id$
  for every $su$-cycle $P$.
  So in this case condition (b1) in Theoren \ref{sufficient} is \eqref{trivial H}, and
condition (b2) can be rewritten as 
$$
\B_{x_0}=\P_{x_0}^{-1}  \circ (\H^{\A,\tilde P}_{{x_0},f{x_0}} )^{-1}\circ \A_{x_0} \circ \P_{x_0}
\;\text{ for some path $\tilde P=\tilde P_{x_0, fx_0}$}.
$$
We choose any $\P_{x_0}\in \drm$, for example $\P_{x_0}=\Id$,  and  define 
a constant cocycle $\B \equiv \B_{x_0}$. Then if follows by Theorem \ref{sufficient}(b) that $\A$ 
is conjugate to  $\B$ via a bounded function  $\P:X\to \drm$ such that 
$\P:X\to \dpm$ is continuous. Also, $\P$ intertwines the holonomies of $\A$ and $\B$, 
which in the case of constant $\B$ means \eqref{PP}.

We note that in this construction a constant cocycle $\B$  is determined by the choice of $\P_{x_0}$
and does not depend on $\tilde P$ by the assumption \eqref{trivial H}.
\vskip.1cm
In the case when $x_0$ is fixed and $\A_{x_0}=\Id$, we obtain $\B\equiv \B_{x_0}=\Id.$


\end{document}